\documentclass[11pt,UTF8]{article}
\usepackage[top=28mm, bottom=28mm, left=35mm, right=35mm]{geometry}

\usepackage{proof}
\usepackage[UKenglish]{babel}
\usepackage{amsmath}
\usepackage{amssymb}
\usepackage{latexsym}
\usepackage{amsthm}
\usepackage{amsfonts}
\usepackage{color}
\usepackage{hyperref}
\usepackage{url}
\usepackage{amsthm}

\hypersetup{
   colorlinks,
   citecolor=blue,
   filecolor=black,
   linkcolor=red,
   urlcolor=black
 }%

\usepackage[all]{xy}
\usepackage{tikz} 
\usepgflibrary{arrows}

\newtheorem{theorem}{Theorem}
\numberwithin{theorem}{section}
\newtheorem{lemma}{Lemma}
\numberwithin{lemma}{section}
\newtheorem{claim}{Claim}
\numberwithin{claim}{section}
\theoremstyle{definition}
\newtheorem{corollary}{Corollary}
\numberwithin{corollary}{section}
\newtheorem{example}{Example}
\numberwithin{example}{section}

\numberwithin{proposition}{section}
\newtheorem{definition}{Definition}
\numberwithin{definition}{section}
\newtheorem{fact}{Fact}
\numberwithin{fact}{section}
\newtheorem{conj}{Conjecture}
\numberwithin{conj}{section}

\def\Ind#1#2{#1\setbox0=\hbox{$#1x$}\kern\wd0\hbox to 0pt{\hss$#1\mid$\hss}
\lower.9\ht0\hbox to 0pt{\hss$#1\smile$\hss}\kern\wd0}
\def\Notind#1#2{#1\setbox0=\hbox{$#1x$}\kern\wd0\hbox to 0pt{\mathchardef
\nn=12854\hss$#1\nn$\kern1.4\wd0\hss}\hbox to
0pt{\hss$#1\mid$\hss}\lower.9\ht0 \hbox to
0pt{\hss$#1\smile$\hss}\kern\wd0}

\title{Pseudofinite Difference Fields and Counting Dimensions}
\author{Tingxiang Zou  (\href{mailto:tingxiang.zou@mail.huji.ac.il}{\url{tingxiang.zou@mail.huji.ac.il}})}
\date{}
\begin{document}
\maketitle
\begin{abstract}
We study a family of ultraproducts of finite fields with the Frobenius automorphism in this paper. Their theories have the strict order property and TP2. But the coarse pseudofinite dimension of the definable sets is definable and integer-valued. Moreover, we also discuss the possible  connection between coarse dimension and transformal transcendence degree in these difference fields.
\end{abstract}

\section{Introduction}
In the development of modern model theory, there is a rich literature devoted to the study of pseudofinite structures. Since they are asymptotic limits of finite structures, their model-theoretic properties often reveal asymptotic behaviours of the corresponding finite classes via \L o\'{s}'s theorem. 
In \cite{chatzidakis1992definable}, a notion of counting measure and dimension of definable sets in pseudofinite fields was developed using the Lang-Weil estimate. It inspired the definition of \emph{one-dimensional asymptotic classes}, a general framework on classes of finite structures based on counting dimension and measure proposed in
 \cite{macpherson2008one}. The ultraproducts of these classes turned out to be model-theoretic tame structures. This counting approach has been further investigated in 
\cite{hrushovski2008counting} and 
\cite{hrushovski2013pseudo} in full generality without any tameness assumptions. 
Two important pseudofinite dimensions have been developed there: the \emph{fine pseudofinite dimension} which comes with a measure (they are the dimension and measure in one-dimensional asymptotic classes) and \emph{coarse pseudofinite dimension}. As has shown in 
 \cite{garcia2015pseudofinite},
 theories with well-behaved fine pseudofinite dimension are tame. Moreover, Hrushovski discovered a surprising link between model theory and finite approximate subgroups using the measure equipped with fine pseudofinite dimension in his fundamental work about stabilizer theorems in 
 \cite{hrushovski2012stable}. And in \cite{hrushovski2013pseudo}, 
Hrushovski again established various links between these two dimensions and algebraic properties of the underlying sets under the assumption of the presence of a field. Connections between model theory of pseudofinite structures and additive combinatorics, e.g. the sum-product phenomenon, the Szemer\'{e}di-Trotter Theorem, have also been made in the same paper.  Recently, significant progress has been made following Hrushovski's approach, for example, a generalization of the Elekes-Szab\'{o} Theorem has been presented using the coarse pseudofinite dimension in 
 \cite{BaysBreuillard}.
 
On the other hand, the class of various expansions of fields is one of the key objects of study in model theory. Examples are differentially closed fields, Henselian valued fields, algebraically closed fields with a generic automorphism, etc. There are lots of natural examples of such structures that are intensively investigated in other areas of mathematics. Studying the model theory of them often extends well-known results to a wider context and sometimes, model-theoretic techniques can help to discover new phenomena. For example, the theory of differentially closed fields plays an important role in Hrushovski's proof of the Mordell-Lang conjecture \cite{hrushovski1996mordell}.

In this paper, we will consider a particular expansion of fields which are pseudofinite: the pseudofinite difference fields, i.e.\ difference fields that are elementary equivalent to ultraproducts of finite fields expanded with some power of Frobenius automorphism. The model theory of pseudofinite fields has been initiated by J. Ax in \cite{ax1968elementary} and subsequently developed in \cite{duret}, \cite{chatzidakis1992definable}, \cite{hrushovski1994groups}. Moreover, the model theory of fields with a distinguished automorphism has also been investigated. The best understood one is ACFA: the theory of algebraically closed fields with a generic automorphism, developed notably in \cite{chatzidakis1999model}, \cite{chatzidakis2002model}. It is the model companion of the theory of difference fields and is model-theoretically tame: supersimple of SU-rank $\omega$. Interestingly, the fixed field of any model of ACFA is a pseudofinite field. Based on these, one might expect a theory of pseudofinite difference fields which is a mixture of PSF (the theory of pseudofinite fields) and ACFA.

M. Ryten studied a specific class of pseudofinite difference fields with the motivation of understanding the asymptotic behaviour of Suzuki groups and Ree groups. In \cite{ryten2007model}, he showed that given any prime $p$ and a pair of coprime numbers $m,n>1$, the class $\{(\mathbb{F}_{p^{k\cdot m+n}},\mbox{Frob}_{p^k}):k\in \mathbb{N}\}$ is a one-dimensional asymptotic class, where $\mbox{Frob}_{p^k}$ is the $k^{th}$ power of the Frobenius map, i.e.\ $\mbox{Frob}_{p^k}$ is the map $x\mapsto x^{p^k}$. He also gave a recursive axiomatization of asymptotic theories of such structures: $\mbox{PSF}_{(m,n,p)}$. In a sense, $\mbox{PSF}_{(m,n,p)}$ is a mixture of PSF and ACFA. In fact, any model of $\mbox{PSF}_{(m,n,p)}$ can be obtained as a definable substructure of some model of ACFA\footnote{See \cite[Lemma 3.3.6]{ryten2007model}.}, and the one-dimensional asymptotic class result is based on the uniform estimate of the number of solutions of definable sets of finite $\sigma$-degree in some model of ACFA in \cite{ryten2006acfa}.

However, $\mbox{PSF}_{(m,n,p)}$ is a bit restricted in the sense that in models of $\mbox{PSF}_{(m,n,p)}$ there are no transformally transcendental elements, i.e.\ elements that satisfy no non-trivial difference polynomial. And most of the nice model-theoretic properties of $\mbox{PSF}_{(m,n,p)}$ come from the tameness of ACFA. 
Our aim in this paper is to study a class of pseudofinite difference fields with transformally transcendental elements. 

Another class of closely related structures is the class of pairs of pseudofinite fields, as the fixed field of a pseudofinite difference field is finite or pseudofinite. As noticed by Macintyre and Cherlin, there are pairs of pseudofinite fields whose theory is not decidable. This wild phenomenon also occurs in the structures that we study. In fact, we will show that in some ultraproduct of finite difference fields there is a definable set such that the family of all internal subsets of it is uniformly definable, see Theorem \ref{thm-sop}. This means in particular that the fine pseudofinite dimension behaves badly and the theory fails to possess tame model-theoretic properties either in the sense of Shelah's classification theory or being decidable, see Corollary \ref{cor-interval}.\footnote{This does not mean that any theory of pseudofinite difference fields with transformally transcendental elements is not tame. We think it is possible that some of them have a decidable theory. But it is not clear which classes and what kind of theories they should be.} However, if we allow the size of the underlying field to grow rapidly enough, then the coarse pseudofinite dimension with respect to the full field behaves extremely well. It takes values in the integers and given a family of uniformly definable sets and an integer $n$, the set of parameters such that the coarse dimension of the corresponding definable sets have value $n$ is definable, see Corollary \ref{cor-main}. This coarse dimension of a definable set in difference fields essentially comes from the fine dimension in pseudofinite fields, which is the Zariski-dimension. Along the line of studying the interaction between counting dimensions and algebraic properties of the underlying structures, we investigate the relation between the integer-valued coarse dimension in our classes of pseudofinite difference fields and the transformal transcendence degree in the algebraic closure. We prove that assuming Conjecture \ref{conj}, i.e.\ if these two dimensions are the same for quantifier-free formulas,  then they also coincide for existential formulas, see Theorem \ref{thm02}. We also classify existentially definable subgroups of algebraic groups under the same conjecture, see Theorem \ref{cor-defgp}.

We remark here that we aim to study the theory of pseudofinite difference fields, which is different with, though closely related to, the theory of pseudofinite fields with a distinguished automorphism. Since there is the concern that the latter may not have a model companion,\footnote{It was claimed that it does not have a model companion in for example \cite[section 3]{chatzidakis1998generic}, but there are some obstacles see \cite[1.12]{chatzidakis2015model}.} neither of these two theories has been carefully studied.

The rest of this paper is organized as the following. Section \ref{sec2} starts with a quick recap of coarse pseudofinite dimension, followed by the definition of a class of ultraproducts of finite difference fields $\mathcal{S}$. The main result is Theorem \ref{th1} and Corollary \ref{cor-main} which states that for any pseudofinite difference field in $\mathcal{S}$, the coarse dimension with respect to the full field $\pmb{\delta}_F$ is integer-valued and definable. Section \ref{sec3} studies the relation between $\pmb{\delta}_F$ and the transformal transcendence degree and it's application to definable groups. The main results are Theorem \ref{thm02} and Theorem \ref{cor-defgp}. Section \ref{sec4} studies the negative model-theoretic aspects of structures of $\mathcal{S}$. They do not belong to any well-studied tame class, is not decidable (Corollary \ref{cor-interval}) and the model-theoretic algebraic closure is different from the algebraic closure in the sense of difference algebra (Theorem \ref{lem-defclosure}). 

\noindent\emph{Notations:} We will denote by $\varphi,\psi,\phi,\ldots$ formulas (possibly with parameters), $x,y,z,\ldots$ tuples of variables, $|x|$ the length of the tuple $x$ and $|\varphi|$ the length of the formula $\varphi$. Suppose $M$ is an $\mathcal{L}$-structure and $\varphi(x)$ an $\mathcal{L}$-formula with parameters in $M$. We write $\varphi(M^{|x|})$ to be the definable set defined by $\varphi(x)$ in $M$, i.e.\ $\varphi(M^{|x|}):=\{a\in M^{|x|}:M\models \varphi(a)\}$. We denote by $\mathbb{P}$ the set of prime numbers.

\textbf{Acknowledgement:} T. Zou is supported by the Golda Meir Fellowship Fund and the European Research Council (ERC) under the European Unions Horizon 2020 research and innovation program (Grant No. 692854). She was partially supported by ValCoMo (ANR-13-BS01-0006) while the research was done. She wants to thank her Ph.D. supervisor Frank Wagner for initiating this interesting project and contributing lots of valuables ideas. She is very grateful to Ehud Hrushovski for pointing out a mistake in the previous version and suggesting the proof of Lemma \ref{lem-conj}. She also wants to thank the referee for numerous comments and suggestions, as well as the encouragement of rewrite the introduction. She is grateful to Zo\'{e} Chatzidakis for answering various questions about pseudofinite fields and ACFA and to Dar{\'i}o Garc{\'i}a for suggestions of corrections in the previous version of this paper.

\section{Coarse pseudofinite dimension}\label{sec2}
We will study the coarse pseudofinite dimension of a class of ultraproducts of finite difference fields in this section. We will show that their coarse dimension with respect to the full field behaves well. The main tool is that the fine dimension of pseudofinite fields is integer-valued and there are only finitely many possible values of the measure for a uniformly definable family of sets of a fixed dimension (see Fact \ref{fact0}). This allows us to estimate the size of sets defined by difference formulas in certain finite difference fields. We show further that the coarse dimension is definable, with only the assumptions that the dimension is integer-valued and a field structure is included in the language. 

We begin with some preliminaries on difference fields and pseudofinite fields.

\begin{definition}
A \textsl{difference field} is a field $(F,+,\cdot,0,1)$ together with a field automorphism $\sigma$ ( in particular $\sigma$ is surjective).

The \textsl{language of difference rings} $\mathcal{L}_\sigma$ is the language of rings augmented by a unary function symbol $\sigma$. 
\end{definition}

\begin{definition} We fix an ambient difference field $L$.
\begin{itemize}
\item
Let $A$ be a subset. We denote by $A_\sigma$ the smallest difference subfield containing $A$ and closed under $\sigma$ and $\sigma^{-1}$. 
\item
Let $E$ be a difference subfield and $a$ be a tuple. The \textsl{$\sigma$-degree}, $\textsf{deg}_\sigma(a/E)$, is the  transcendence degree of $(E,a)_{\sigma}$ over $E$. 
\item
Let $E$ be a difference subfield. If there is no non-zero difference polynomial over $E$ vanishing on $a$, then we say $a$ is \textsl{transformally transcendental} over $E$ if $a$ is an element in $L$ and $a$ is \textsl{transformally independent} over $E$ if $a$ is a tuple in $L$.
\item
Let $E$ be a difference subfield and $a$ be a tuple. The \textsl{transformal transcendence degree} of $a$ over $E$ is defined as the maximal length of a transformally independent subtuple of $a$ over $E$.
\end{itemize}
\end{definition}

We now give the definition of pseudofinite structures and coarse pseudofinite dimension.

\begin{definition} By a \emph{pseudofinite structure} we mean an infinite structure that is elementary equivalent to a non-principal ultraproduct of finite structures.
\end{definition}
\noindent\emph{Remark:} In this paper, we assume that we always work with the ultraproducts, as they are essential for the definition of coarse pseudofinite dimension. Hence, from now on, by a pseudofinite structure, we mean an infinite ultraproduct of finite structures. Let $M=\prod_{i\in I}M_i/\mathcal{U}$ be a pseudofinite structure. We say a set $X\subseteq M^n$ is \emph{internal} if $X=\prod_{i\in I}X_i/\mathcal{U}$ where $X_i\subseteq (M_i)^n$ for each $i\in I$.

\begin{definition}
Let $M$ be a pseudofinite structure over some non-principal ultrafilter $\mathcal{U}$ on $I$ and $\mathbb{R}^*$ be the ultrapower of $\mathbb{R}$ along $\mathcal{U}$. Then any internal set $D\subseteq M^n$ has a non-standard cardinality $|D|\in\mathbb{R}^*$, that is, if $D:=\prod_{i\in I}D_i/\mathcal{U}$, then $|D|:=(|D_i|)_{i\in I}/\mathcal{U}\in\mathbb{R}^*$, where $|D_i|$ is the cardinality of the finite set $D_i$. Let $\alpha\in\mathbb{R}^*$. 

\begin{itemize}
\item
The \textsl{coarse pseudofinite dimension} (or simply \emph{coarse dimension}) on $M$ normalised by $\alpha$, denoted by $\pmb{\delta}_{\alpha}$, is a function from definable sets of $M$ to $\mathbb{R}^{\geq0}\cup\{\infty\}$, defined as $$\pmb{\delta}_{\alpha}(A):=\mbox{st}\left(\frac{\log|A|}{\alpha}\right),$$ for $A\subseteq M^n$ definable. When $\alpha:=\log|X|$ for some internal set $X$, we also write $\pmb{\delta}_\alpha$ as $\pmb{\delta}_{X}$ and call $\pmb{\delta}_{X}$ the \textsl{coarse pseudofinite dimension with respect to $X$}.

If $\varphi(x)$ is a formula with parameters in $M$, then we also write $\pmb{\delta}_{X}(\varphi(M^{|x|}))$ as $\pmb{\delta}_{X}(\varphi(x))$.
\item
We say $\pmb{\delta}_{\alpha}$ is \textsl{continuous} if for any formula $\phi(x,y)$ without parameters and for any $r_1<r_2\in \mathbb{R}$, there is some $\emptyset$-definable set $D$ with $$\{a\in M^{|y|}:\pmb{\delta}_{\alpha}(\phi(x,a))\leq r_1\}\subseteq D\subseteq \{a\in M^{|y|}:\pmb{\delta}_{\alpha}(\phi(x,a))<r_2\}.$$
\item
We say $\pmb{\delta}_{\alpha}$ is \textsl{definable} if $\pmb{\delta}_{\alpha}$ is continuous and the set $\{\pmb{\delta}_{\alpha}(\phi(x,a)):a\in M^{|y|}\}$ is finite for any $\emptyset$-definable formula $\phi(x,y)$. By compactness, it is equivalent to the following: for any $\emptyset$-definable formula $\phi(x,y)$ and $a\in M^{|y|}$, there is $\xi(y)\in\mbox{tp}(a)$ such that $$M\models\xi(b)\text{ if and only if }\pmb{\delta}_{\alpha}(\phi(x,b))=\pmb{\delta}_{\alpha}(\phi(x,a)).$$
\end{itemize}

\end{definition}

\begin{definition}
Let $M$ be a pseudofinite structure and $\alpha\in \mathbb{R}^*$. Let $\pi(x)$ be partial type. Define $$\pmb{\delta}_{\alpha}(\pi(x)):=\inf\left\{\pmb{\delta}_{\alpha}(\varphi(x)),\pi(x)\models \varphi(x)\right\}.$$
Let $a$ be a tuple in $M$ and $A\subseteq M$. 
Define $\pmb{\delta}_{\alpha}(a/A):=\pmb{\delta}_{\alpha}(\mbox{tp}(a/A)).$
\end{definition}
\begin{fact}\cite[Lemma 2.10]{hrushovski2013pseudo} If $\pmb{\delta}_{\alpha}$ is continuous, then $\pmb{\delta}_{\alpha}$ is additive, i.e.\ for any $a,b,A\subseteq M$ we have $\pmb{\delta}_{\alpha}(a,b/A)=\pmb{\delta}_{\alpha}(a/A,b)+\pmb{\delta}_\alpha(b/A)$.
\end{fact}

\noindent \emph{Remark:} 
\begin{itemize}
\item
If $\pi$ is a partial type over $A$, then $\pi$ extends to a complete type $p$ over $A$ with $\pmb{\delta}_\alpha(\pi)=\pmb{\delta}_\alpha(p)$.
\item
\cite[5.1]{hrushovski2012stable} $\pmb{\delta}_{\alpha}$ satisfies subadditivity in general, namely, if $f:X\to Y$ is a definable function such that $\pmb{\delta}_{\alpha}(f^{-1}(y))\geq\alpha$ for any $y\in Y$ and $\pmb{\delta}_{\alpha}(Y)\geq\beta$, then $\pmb{\delta}_{\alpha}(X)\geq\alpha+\beta$.
\item
There is always a way to make $\pmb{\delta}_{\alpha}$ continuous by expanding the language of the structure $M$. However, this might add new definable sets to $M$, which could be an inconvenience.
\end{itemize}

The following fact is a well-known result in the class of finite fields, which gives a uniform estimate of number of solutions of definable sets in all finite fields. Our main result will be based on it.

\begin{fact}\label{fact0}\cite[Main Theorem]{chatzidakis1992definable}
Let $\mathcal{L}$ be the language of rings. For every formula $\varphi(x,y)\in\mathcal{L}$ with $|x|=n,|y|=m$ there are a constant $C_{\varphi}>0$ and a finite set $D_{\varphi}\subset \{0,\ldots,n\}\times\mathbb{R}^{>0}$ such that the following holds:

For any finite field $\mathbb{F}_q$ and $a\in (\mathbb{F}_q)^m$, if $\varphi((\mathbb{F}_q)^n,a)\neq\emptyset$, then there is some $(d,\mu)\in D$ such that $$\left|\left|\varphi\left((\mathbb{F}_q)^n,a\right)\right|-\mu\cdot q^d\right|\leq C_{\varphi}\cdot q^{d-\frac{1}{2}}.$$
\end{fact}

Now we start to define a special class of ultraproducts of finite difference fields and study their coarse pseudofinite dimension with respect to the full field. The main observation is that given a difference formula $\varphi(x)$ and we want to estimate the size of the set that $\varphi(x)$ defines in a finite difference field $(\mathbb{F}_{p^k},\mbox{Frob}_{p^m})$. If we allow $k$ grow while keep $p$ and $m$ fixed, then the set defined by $\varphi(x)$ has a dimension which comes from the fine pseudofinite dimension in the classes of pseudofinite fields. The trick is that we translate the difference formula $\varphi(x)$ into a ring formula $\varphi_{p^m}(x)$ by replacing terms $\sigma(t)$ with $t^{p^m}$. If $k$ is big enough compared to $p$ and $m$, then the set defined by $\varphi(x)$ in $(\mathbb{F}_{p^k},\mbox{Frob}_{p^m})$ will be roughly propositional to $(p^k)^d$, where $d\leq |x|$ is the fine dimension of $\varphi_{p^m}$, which depends on $\varphi$, $p$ and $m$. If we take an ultraproduct of $\{(\mathbb{F}_{p^k},\mbox{Frob}_{p^m}):p\in\mathbb{P},k,m\geq 1\}$ over some non-principal ultrafilter $\mathcal{U}$, then $\mathcal{U}$ will pick one of the dimension $d\leq |x|$. Suppose almost all $k$ in $(\mathbb{F}_{p^k},\mbox{Frob}_{p^m})$ are big enough compared to $p$ and $m$, then $d$ will be the coarse pseudofinite dimension with respect to the full field of the set defined by $\varphi$ in the ultraproduct.

\begin{definition}\label{def-phip}
Let $\mathcal{L}_\sigma$ be the language of difference rings. Let $\varphi(x,y)$ be a formula defined in $\mathcal{L}_\sigma$ without parameters.
For any prime $p$, define $\varphi_p(x,y)$ as the result of replacing each occurrence of $\sigma(t)$ in $\varphi(x,y)$ by $t^{p}$. Clearly, $\varphi_p(x,y)$ is a formula in the language of rings $\mathcal{L}$.
\end{definition}

Let $\mathbb{P}$ be the set of all primes. For any formula $\varphi(x,y)$ in $\mathcal{L}_\sigma$ and $p\in\mathbb{P}$, consider $\varphi_p(x,y)\in\mathcal{L}$. There are $C_{\varphi_p}$ and the finite  set $D_{\varphi_p}$ as stated in Fact \ref{fact0}. Let $$E_{\varphi_p}:=\bigcup_{0\leq d\leq |x|}\{\mu:(d,\mu)\in D_{\varphi_p}\}.$$

Define $$N^p_{\varphi(x,y)}:=\max\left\{\mu,\frac{1}{\mu},2\log_p\left(\frac{2C_{\varphi_p}}{\mu}\right):\mu\in E_{\varphi_p}\right\}.$$

Let \begin{equation}\label{eq1}
f(\ell,p):=\max\{N^p_{\varphi(x,y)}:|\varphi(x,y)|\leq \ell\}.
\end{equation}

\begin{definition}
Define the family $\mathcal{S}$ of pseudofinite difference fields as $$\mathcal{S}:=\left\{\prod_{p\in \mathbb{P}}(\mathbb{F}_{p^{k_p}},\text{Frob}_{p})/\mathcal{U}:k_p\geq f(p,p) \text{ for all }p\in\mathbb{P}, ~\mathcal{U} \text{ a non-principal ultrafilter}\right\}.$$
\end{definition}

\begin{theorem}\label{th1}
Let $(F,\mbox{Frob}):=\prod_{p\in \mathbb{P}}(\mathbb{F}_{p^{k_p}},\mbox{Frob}_{p})/\mathcal{U}\in \mathcal{S}$. Then the coarse pseudofinite dimension with respect to $F$ is integer-valued on all $\mathcal{L}_\sigma$-definable sets. 
\end{theorem}
\begin{proof}
Let $\varphi(x,y)$ be an $\mathcal{L}_\sigma$-formula. Consider a parameter $a=(a_p)_{p\in\mathbb{P}}/\mathcal{U}\in F^{|y|}$. For any $p\in \mathbb{P}$, we know that there are $(d_{k_p},\mu_{k_p})\in \{0,\ldots,|x|\}\times \mathbb{R}^{>0}$ and $C_{\varphi_{p}}\geq 0$ such that for $a_p\in (\mathbb{F}_{p^{k_p}})^{|y|}$, we have $$\left|\left|\varphi_p\left((\mathbb{F}_{p^{k_p}})^{|x|},a_p\right)\right|-\mu_{k_p}\cdot p^{k_p\cdot d_{k_p}}\right|\leq C_{\varphi_p}\cdot p^{k_p(d_{k_p}-\frac{1}{2})}.$$ We say that $\varphi_p(x,a_p)$ has dimension $d_{k_p}$ in $\mathbb{F}_{p^{k_p}}$. As $d_{k_p}\leq |x|$, there is exactly one $d\in\{0,\ldots,|x|\}$ with $\left\{p\in\mathbb{P}: \varphi_p(x,a_p) \text{ has dimension } d\text{ in }\mathbb{F}_{p^{k_p}}\right\}\in \mathcal{U}$. We claim that $\pmb{\delta}_F(\varphi(F^{|x|},a))=d$.

Proof of the claim: Note that for any $p\in\mathbb{P}$ and $c\in (\mathbb{F}_{p^{k_p}})^{|x|}$, we have $$\mathbb{F}_{p^{k_p}}\models\varphi_p(c,a_p)\text{ if and only if } (\mathbb{F}_{p^{k_p}},\mbox{Frob}_p)\models \varphi(c,a_p).$$ 

Let $I=\left\{p\in\mathbb{P}:p>|\varphi(x,y)|\text{ and }\varphi_p(x,a_p) \text{ has dimension } d \text{ in }\mathbb{F}_{p^{k_p}}\right\}$. Clearly, $I\in\mathcal{U}$. Then for any $p\in I$, $$\left|\left|\varphi_p\left((\mathbb{F}_{p^{k_p}})^{|x|},a_p\right)\right|-\mu_{k_p}\cdot p^{k_p\cdot d}\right|\leq C_{\varphi_p}\cdot p^{k_p(d-\frac{1}{2})},$$ and $k_p\geq f(p,p)\geq \max\left\{\mu_{k_p},\dfrac{1}{\mu_{k_p}},2\log_p\left(\dfrac{2C_{\varphi_p}}{\mu_{k_p}}\right)\right\}$.

As $k_p\geq 2\log_p\left(\dfrac{2C_{\varphi_p}}{\mu_{k_p}}\right)$, we get $$C_{\varphi_p}\cdot p^{k_p(d-\frac{1}{2})}\leq \frac{1}{2}\mu_{k_p}\cdot p^{k_p\cdot d}.$$ Therefore, $$\frac{1}{2}\mu_{k_p}\cdot p^{k_p\cdot d}\leq \left|\varphi_p\left((\mathbb{F}_{p^{k_p}})^{|x|},a_p\right)\right|\leq \frac{3}{2}\mu_{k_p}\cdot p^{k_p\cdot d}.$$ Furthermore, by the definition of $k_p$, we have $\frac{1}{k_p}<\mu_{k_p}<k_p$. Hence, $$\frac{1}{2k_p}\cdot p^{k_p\cdot d}\leq\left|\varphi_p\left((\mathbb{F}_{p^{k_p}})^{|x|},a_p\right)\right|\leq 2k_p\cdot p^{k_p\cdot d}.$$ This implies $$d-\frac{\log(2k_p)}{k_p\cdot \log p}\leq\frac{\log\left|\varphi_p\left((\mathbb{F}_{p^{k_p}})^{|x|},a_p\right)\right|}{\log (p^{k_p})}\leq d+\frac{\log(2k_p)}{k_p\cdot\log p }.$$ Since $\lim_{p\to\infty}\dfrac{\log(2k_p)}{k_p\cdot \log p}=0$, we have $$\lim_{p\to\infty,~p\in I}\dfrac{\log\left|\varphi_p\left((\mathbb{F}_{p^{k_p}})^{|x|},a_p\right)\right|}{\log (p^{k_p})}=d.$$ 

Therefore, $\pmb{\delta}_F(\varphi(x,a))=d$.
\end{proof}

\noindent \emph{Remark:} This proof works also for pseudofinite difference fields of characteristic $p>0$, that is, for $\prod_{i\in I}(\mathbb{F}_{p^{k_i}},\mbox{Frob}_{p^{m_i}})/\mathcal{U}$ provided $k_i>>m_i$ for almost all $i$. More precisely, in the proof of Theorem \ref{th1}, instead of translating $\varphi$ to $\varphi_{p}$ for each prime $p$, we translate it to $\varphi_{p^{m_i}}$ for each $i\in I$. That is, given a difference formula $\varphi(x,y)$ we consider the following ring formula $\varphi_{p^{m_i}}(x,y)$ obtained by replacing each occurrence of $\sigma(t)$ in $\varphi(x,y)$ by $t^{p^{m_i}}$. Then we use Fact \ref{fact0} and the same strategy to get the desired result.

In the following, we will show that the coarse dimension $\pmb{\delta}_{F}$ is definable using the field structure. To prove this, we first need a lemma.

\begin{lemma}\label{lem-subadd}
Let $M$ be an ultraproduct of finite structures in the language $\mathcal{L}'$ and $X$ be an internal subset of $M$. Let $\varphi(x,y)$ be an $\mathcal{L}'$-formula with $|x|=m$ and $|y|=n$. Suppose there is some $r\in \mathbb{R}^{\geq 0}$ such that for all $b\in M^m$ we have $\pmb{\delta}_X(\varphi(x,b))=r$ whenever $\varphi(x,b)\neq\emptyset$. Then $$\pmb{\delta}_X(\varphi(x,y))=r+\pmb{\delta}_X(\exists x\varphi(x,y)).$$
\end{lemma}
\begin{proof}
Suppose $(M,X)=\prod_{i\in I}(M_i,X_i)/\mathcal{U}$ for some ultrafilter $\mathcal{U}$ on an index set $I$ and $X_i\subseteq M_i$ finite sets. For each $i\in I$ pick $b_i^{max}$ and $b_i^{min}$ in $(M_i)^m$ such that $|\varphi((M_i)^n,b_i^{max})|$ is maximal and $|\varphi((M_i)^n,b_i^{min})|$ is minimal non-zero respectively. Clearly, we have $$|\varphi((M_i)^n,b_i^{min})|\cdot |\exists x\varphi(x,(M_i)^m)|\leq|\varphi((M_i)^{n+m})|\leq |\varphi((M_i)^n,b_i^{max})|\cdot |\exists x\varphi(x,(M_i)^m)|.$$ Let $b^{max}:=(b_i^{\max})_{i\in I}/\mathcal{U}\in M$ and $b^{min}:=(b_i^{min})_{i\in I}/\mathcal{U}\in M$ respectively. By assumption, $\pmb{\delta}_X(\varphi(x,b^{max}))=\pmb{\delta}_X(\varphi(x,b^{min}))=r$. Therefore, for any $\epsilon>0$, there is some $J\in\mathcal{U}$ such that for all $i\in J$, we have $$|X_i|^{r-\epsilon}\leq |\varphi((M_i)^n,b_i^{min})|\leq |\varphi((M_i)^n,b_i^{max})|\leq |X_i|^{r+\epsilon}.$$ Multiplying each term by $|\exists x\varphi(x,(M_i)^m)|$ and combining the inequality before, we get $$|X_i|^{r-\epsilon}\cdot |\exists x\varphi(x,(M_i)^m)|\leq\varphi((M_i)^{n+m})\leq |X_i|^{r+\epsilon}\cdot |\exists x\varphi(x,(M_i)^m)|.$$ Therefore, $$r-\epsilon +\frac{\log |\exists x\varphi(x,(M_i)^m)|}{\log |X_i|}\leq \frac{\log |\varphi((M_i)^{n+m})|}{\log |X_i|}\leq r+\epsilon +\frac{\log |\exists x\varphi(x,(M_i)^m)|}{\log |X_i|}.$$ By the definition of $\pmb{\delta}_X$ we conclude that $$r+\epsilon+\pmb{\delta}_X(\exists x\varphi(x,y))\leq \pmb{\delta}_X(\varphi(x,y))\leq r-\epsilon+\pmb{\delta}_X(\exists x\varphi(x,y)).$$ Since $\epsilon$ is arbitrary, we get the desired result.
\end{proof}

\begin{corollary}\label{cor-int-delta}
Let $M$ be a pseudofinite structure in the language $\mathcal{L}$ and let $X\subseteq M^n$ be an internal set. Suppose there is some $r\in\mathbb{N}$ such that for any $\mathcal{L}$-formula $\varphi(x,y)$ with $|x|=1$ over $\emptyset$ and any $b\in M^{|y|}$, we have $\pmb{\delta}_{X}(\varphi(x,b))\in\{0,1,\ldots,r\}$ and for each $i\leq r$, the set $$\{b\in M^{|y|}:\pmb{\delta}_{X}(\varphi(x,b))=i\}$$ is $\emptyset$-definable. Then for any formula $\psi(x,y)$ and any tuple $c\in M^{|y|}$, we have $$\pmb{\delta}_X(\psi(x,c))\in\{0,\ldots,|x|\cdot r\}.$$ Moreover, $\pmb{\delta}_{X}$ is definable.
\end{corollary}
\begin{proof}
We use induction on the length of $|x|.$ The case $|x|=1$ is given by assumption. 

Suppose the conclusion holds for $|x|=n$, we prove it for $|x|=n+1$. Let $\psi(x_0,\ldots,x_n,y)$ be a formula with $|x_i|=1$ for $0\leq i\leq n$. We know that there are formulas without parameters $\theta_\ell(x_1,\ldots,x_n,y)$ for $\ell\in\{0,1,\ldots,r\}$ which define respectively the sets $$\{(a_1,\ldots,a_n,b)\in M^{n+|y|}: \pmb{\delta}_M(\psi(x_0,a_1,\ldots,a_n,b))=\ell\text{ and }\psi(M,a_1,\ldots,a_n,b)\neq\emptyset\}.$$ 

For any $c\in M^{|y|}$, note that $\psi(M^{n+1},c)$ is the disjoint union of $$\{\psi(M^{n+1},c)\land\theta_{\ell}(M^{n},c):\ell\in\{0,1,\ldots,r\}\},$$ and Lemma \ref{lem-subadd} applies to each of these formulas. Hence, 
\begin{align*}
\pmb{\delta}_X(\psi(x_0,x_1,\ldots,x_n,c)\land\theta_{\ell}(x_1,\ldots,x_n,c))\\
=\ell+\pmb{\delta}_X(\exists x_0(\psi(x_0,x_1,\ldots,x_n,c)\land\theta_{\ell}(x_1,\ldots,x_n,c))\\
=\ell+\pmb{\delta}_X(\theta_{\ell}(x_1,\ldots,x_n,c)).
\end{align*}

By induction hypothesis, $\pmb{\delta}_X(\theta_{\ell}(x_1,\ldots,x_n,c))\in\{0,\ldots,r\cdot n\}$. Therefore, $$\pmb{\delta}_X(\psi(x_0,x_1,\ldots,x_n,c))=\max\{\ell+\pmb{\delta}_X(\theta_{\ell}(x_1,\ldots,x_n,c)):0\leq \ell\leq r\}\in\{0,\ldots,r\cdot(n+1)\}.$$

Again by induction hypotheses, for any $k\in\{0,\ldots,r\cdot n\}$ there are $\emptyset$-definable $\xi_{\ell}^k(y)$ with $\ell\in\{0,\ldots,r\}$, which define the corresponding sets $$\{b\in F^{|y|}:\pmb{\delta}_X(\theta_{\ell}(x_1,\ldots,x_n,b))= k\text{ and }\theta_{\ell}(M^{n},b)\neq\emptyset\}.$$ Then the formula $\displaystyle{\bigvee_{0\leq \ell\leq r,~0\leq j\leq r\cdot n,~\ell+j=t}\xi_{\ell}^j(y)}$ defines the set $$\{b\in M^{n+1}:\pmb{\delta}_M(\psi(x_0,x_1,\ldots,x_n,b))= t\text{ and }\psi(M^{n+1},b)\neq\emptyset\}$$ for any $t\in\{0,\ldots,r\cdot(n+1)\}$.
\end{proof}

\begin{lemma}\label{lem-deffield}
Let $\mathcal{M}=(F,+,\cdot,0,1,\ldots)$ be a pseudofinite field with some extra structures. Let $\pmb{\delta}_F$ be the coarse pseudofinite dimension normalised by $|F|$. Suppose for any formula $\varphi(x,y)$ with $|x|=1$ we have $\pmb{\delta}_F(\varphi(x,b))\in\{0,1\}$ for any tuple $b\in F^{|y|}$. Then $\pmb{\delta}_F$ is definable and for any formula $\psi(x,y)$ and any tuple $c\in F^{|y|}$, we have $\pmb{\delta}_F(\psi(x,c))\in\{0,\ldots,|x|\}$. 
\end{lemma}

\begin{proof}
By Corollary \ref{cor-int-delta}, we only need to show definability when $|x|=1$. 

For each $\psi(x,y)$, consider the formula $$\theta_{\psi}(y):=\forall z\exists x_1\exists x_2\exists x_3\exists x_4~\left(\bigwedge_{1\leq i\leq 4}\psi(x_i,y)~\land~ x_3\neq x_4~\land~ z=(x_1-x_2)\cdot(x_3-x_4)^{-1}\right).$$ We claim that $\theta_{\psi}(c)$ holds if and only if $\pmb{\delta}_F(\psi(x,c))=1$ for all $c\in F^{|y|}$.  Suppose $\theta_{\psi}(c)$ hold. Then there is a map from $(\psi(F,c))^4$ to $F$ defined by sending $(x_1,x_2,x_3,x_4)$ to $(x_1-x_2)(x_3-x_4)^{-1}$ if $x_3\neq x_4$, otherwise we map $(x_1,x_2,x_3,x_4)$ to $0$. The formula $\theta_{\psi}(c)$ holds means exactly that the map is surjective. Therefore, $\pmb{\delta}_F(\psi(x,c))\geq \frac{1}{4}\pmb{\delta}_F(F)=\frac{1}{4}$. By assumption, $\pmb{\delta}_F(\psi(x,c))\in\{0,1\}$. Hence, $\pmb{\delta}_F(\psi(x,c))=1$. On the other hand, if $\neg\theta_{\psi}(c)$ holds, there is $a\in F$ such that for any $x_1,x_2,x_3,x_4\in \psi(F,c)$ we have $a\neq (x_1-x_2)(x_3-x_4)^{-1}$ whenever $x_3\neq x_4$. Let $f:(\psi(F,c))^2\to F$ be defined as $f(x_1,x_2):=x_1+a x_2$. Then $f$ is an injection. Therefore, $\pmb{\delta}_F(\psi(x,c))\leq \frac{1}{2}$. We conclude that $\pmb{\delta}_F(\psi(x,c))=0$. 

Hence, the set $\displaystyle{\{c\in F^{|y|}:\pmb{\delta}_F(\psi(x,c))= 0\text{ and }\psi(F,c)\neq\emptyset\}}$ is defined by $\neg\theta_{\psi}(y)\land\exists x\psi(x,y)$, and $\theta_{\psi}(y)$ defines the set $\displaystyle{\{c\in F^{|y|}:\pmb{\delta}_F(\psi(x,y))= 1\}.}\qedhere$
\end{proof}

\begin{corollary}\label{cor-main}
For any pseudofinite difference field $(F,\mbox{Frob})\in\mathcal{S}$, the coarse dimension $\pmb{\delta}_F$ is definable and integer-valued for all $\mathcal{L}_\sigma$-definable sets. 
Moreover, $\pmb{\delta}_F$ is additive in the language $\mathcal{L}_{\sigma}$.
\end{corollary}
\begin{proof}
By Theorem \ref{th1}, for any $\mathcal{L}_\sigma$-formula $\psi(x,y)$ with $|x|=1$, any $b\in F^{|y|}$ we have $$\pmb{\delta}_F(\psi(x,b))\in\{0,1\}.$$ Applying Lemma \ref{lem-deffield} we get the desired result.
\end{proof}

\noindent \emph{Remark:} In general, the coarse dimension does not have the property that a definable set has dimension 0 if only if it is finite. Similarly, in a pseudofinite group, a subgroup of infinite index does not necessarily have smaller dimension, as we show in the next example.

\begin{example}
Let $(F,\mbox{Frob})=\prod_{p\in\mathbb{P}}(\mathbb{F}_{p^{k_p}},\mbox{Frob}_p)/\mathcal{U}\in\mathcal{S}$. Define a function $f:F^{\times}\to F^{\times}$ as $$f(x):=x^{-1}\cdot\mbox{Frob}(x).$$ It is easy to see that $f$ is a group homomorphism. Therefore, the image $T:=f(F^{\times})$ is a definable subgroup of $F^{\times}$. There is a corresponding $f_p:(\mathbb{F}_{p^{k_p}})^{\times}\to (\mathbb{F}_{p^{k_p}})^\times$ and $T_p:=f_p((\mathbb{F}_{p^{k_p}})^\times)$ for any $p\in \mathbb{P}$. Since the kernel of $f_p$ is $(\mathbb{F}_{p})^{\times}$, we get $[(\mathbb{F}_{p^{k_p}})^\times:T_p]=p-1$. Hence, $T$ has infinite index in $F^\times$, though $\pmb{\delta}_F(T)=\pmb{\delta}_F(F^{\times})$.
\end{example}

\section{Coarse dimension and transformal transcendence degree}\label{sec3}

In the following, we will study some algebraic properties of difference fields that are intrinsic to the coarse dimension $\pmb{\delta}_F$. Our aim is to understand the theory of difference fields in $\mathcal{S}$ in terms of $\pmb{\delta}_F$.

In model theory, we always understand definable sets or definable structures ``up to a finite noise''. For example, strongly minimal theories are considered transparent since every definable subset is either finite or cofinite. And in groups, people always go to a definable subgroup of finite index freely. As we will see in the next section, in any member of $\mathcal{S}$ there is a uniformly definable family of sets of coarse dimension 0 that contains any internal subset of a fixed infinite definable set. Hence, up to a finite noise, the family still has all the wild phenomena that should not appear in a ``nice'' structure. However, it seems that coarse dimension 0 sets are the only true obstacle of tameness. In other words, it is possible that all definable sets and definable structures of $\mathcal{S}$ are tame ``up to a noise of coarse dimension 0''. This section will provide some positive evidence of this point of view. Basically, we want to associate the coarse dimension of a tuple with the transformal transcendence degree of it. And if we can do this, then the quantifier-free type of a tuple will determine the coarse dimension of this tuple, which will imply that for any definable set of dimension $n$, there is a quantifier-free definable set of the same dimension such that their intersection also has dimension $n$. Thus, definable sets can be understood by quantifier-free definable ones ``up to coarse dimension 0''. 

Let us start with an observation. Given $(F,\mbox{Frob})=(\mathbb{F}_{p^{k_p}},\mbox{Frob}_p)/\mathcal{U}\in\mathcal{S}$. Let $$(\tilde{F},\mbox{Frob}):=\prod_{p\in \mathbb{P}}(\tilde{\mathbb{F}}_{p},\mbox{Frob}_p)/\mathcal{U},$$ then by \cite[Theorem 1.4]{hrushovski2004elementary} we have $(\tilde{F},\mbox{Frob})$ is a model of ACFA, which contains $(F,\mbox{Frob})$ as a substructure. 

In ACFA, there is a notion of dimension which is also integer-valued, and it is induced by SU-rank.

\begin{definition}
Let $\mathbf{k}$ be a saturated model of ACFA.
Let $a$ be a finite tuple in $\mathbf{k}$ and $A\subseteq \mathbf{k}$. Then $\mbox{SU}(a/A)=\omega\cdot k+n$ for some $0\leq k\leq |a|$.  Define the \textsl{rank-dimension} $\mbox{dim}_{rk}$ of $\mbox{tp}(a/A)$ as  $\mbox{dim}_{rk}(a/A):=k$. 

For a partial type $\pi(x)$ with parameters $A$. Let 
$$d:=\max\{\mbox{dim}_{rk}(a/A): \mathbf{k}\models\pi(a)\}
=\max\{n\leq |x|:\mbox{SU}(a/A)=\omega\cdot n+m,\mathbf{k}\models\pi(a)\}.$$
We define $\mbox{dim}_{rk}(\pi(x)):=d$.
\end{definition}

\noindent \emph{Remark:} $\mbox{dim}_{rk}(a/A)$ coincides with the transformal transcendence degree of $a$ over $A_\sigma$ (the difference field generated by $A$). By \cite[Chapter 5, Theorem 3]{cohn1965difference}, the transformal transcendence degree of $a$ over $(A_\sigma)^{alg}$ (the difference field algebraic closure) is the same as that over $A_\sigma$. Therefore, $\mbox{dim}_{rk}(a/A)$ is determined by the quantifier-free type of $a$ over $A$.

Now we have two integer-valued additive dimensions on types: the rank-dimension $\mbox{dim}_{rk}$ and the coarse dimension $\pmb{\delta}_F$. It is natural to ask whether they coincide. One of the inequalities is obvious.
\begin{lemma}\label{lem01}
Let $(F,\mbox{Frob})\in\mathcal{S}$. For any tuple $a\in F$ and subset $A\subseteq F$ we have $\pmb{\delta}_F(a/A)\leq\mbox{dim}_{rk}(a/A)$.
\end{lemma}
\begin{proof}
Note that by the additivity of both $\mbox{dim}_{rk}$ and $\pmb{\delta}_F$, we only need to prove the inequality when $a$ is a single element. We may assume that $A=A_{\sigma}$. By \cite{chatzidakis1999model}, we know that $SU(a/A)=\omega$ if and only if $a$ is transformally transcendental over $A$ if and only if $deg_{\sigma}(a/A)=\infty$. Therefore, we need to show that if $deg_{\sigma}(a/A)<\infty$ then $\pmb{\delta}_F(a/A)=0$.

Suppose $deg_{\sigma}(a/A)<\infty$. Then there is some $m$ and a non-trivial polynomial $f(x;y_1,\ldots,y_m)$ with coefficients in $A$, such that  $f(\sigma^m(a);\sigma^{m-1}(a),\ldots,a)=0$. Take any prime $p\in\mathbb{P}$ and let $g_p(x):=f(x^{p^m};x^{p^{m-1}},\ldots,x)$. Then $$|\{a'\in\mathbb{F}_{p^{k_p}}:g_p(a')=0\}|\leq p^{C\cdot m}$$ for some constant $C$ depending on $f$. Let $\varphi(x):=f(\sigma^m(x);\sigma^{m-1}(x),\ldots,x)=0$. Then $\varphi(x)$ defines exactly the set of zeros of $g_p$ in $(\mathbb{F}_{p^{k_p}},\mbox{Frob}_p)$. Therefore, $\pmb{\delta}_F(\varphi(x))=0$. As $a\in \varphi(F)$, we get $\pmb{\delta}_F(a/A)=0$.
\end{proof}

We conjecture that in general the two dimensions coincide. However, even in the case of quantifier-free types, it is not clear to the author. If it were true for quantifier-free types, we are able to extend the equivalence to existential types.

We first state the conjecture for quantifier-free types here:

\begin{conj}\label{conj}
Let $(F,\mbox{Frob})\in\mathcal{S}$. For any countable set $A\subseteq F$ and complete quantifier-free $\mathcal{L}_\sigma$-type $p(x)\in S^{qf}_n(A)$, the following holds $$\pmb{\delta}_F(p(x))=\mbox{dim}_{rk}(p(x)).$$
\end{conj}

In the following we will show that if Conjecture \ref{conj} is true, then the same holds for existential types.

\begin{lemma}\label{lem-conj}
Let $(F,\mbox{Frob})\in\mathcal{S}$ and $A\subseteq F$ a countable set. Let $r(x,y)$ be a complete quantifier-free $\mathcal{L}_\sigma$-type over $A$. Suppose $\pmb{\delta}_{F}(\exists y~\! r(x,y))=0$ then $\mbox{dim}_{rk}(\exists y~\! r(x,y))=0$.
\end{lemma}
\begin{proof}
Let $p(x)$ and $q(y)$ be the restrictions of $r(x,y)$ to $x$ and $y$ variables respectively. Note that since $r(x,y)$ is a complete quantifier-free type, then for any $a,b$ such that $(\tilde{F},\mbox{Frob})\models r(a,b)$, we have $\mbox{dim}_{rk}(r(x,y))=\mbox{dim}_{rk}(a,b/A)$ and $$\mbox{dim}_{rk}(a,b/A)=\mbox{dim}_{rk}(b/a,A)+\mbox{dim}_{rk}(a/A)=\mbox{dim}_{rk}(r(a,y))+\mbox{dim}_{rk}(p(x)).$$
We distinguish two different cases. The first case is that $\mbox{dim}_{rk}(r(x,y))=\mbox{dim}_{rk}(p(x))+\mbox{dim}_{rk}(q(y))$. Then for any $a,b$ such that $(\tilde{F},\mbox{Frob})\models r(a,b)$, we have 
\begin{align*}
\mbox{dim}_{rk}(p(x))+\mbox{dim}_{rk}(q(y))=\mbox{dim}_{rk}(r(x,y))\\
=\mbox{dim}_{rk}(a,b/A)=\mbox{dim}_{rk}(r(a,y))+\mbox{dim}_{rk}(p(x))
\end{align*} We conclude $\mbox{dim}_{rk}(r(a,y))=\mbox{dim}_{rk}(q(y))$ for any $a$ such that there is $b$ with $(\tilde{F},\mbox{Frob})\models r(a,b)$.

The second case is that $\mbox{dim}_{rk}(r(x,y))<\mbox{dim}_{rk}(p(x))+\mbox{dim}_{rk}(q(y))$. Then $$\mbox{dim}_{rk}(r(a,y))<\mbox{dim}_{rk}(q(y))$$ for any $a$ such that there is $b$ with $(\tilde{F},\mbox{Frob})\models r(a,b)$.

In the first case, as $\mbox{dim}_{rk}(r(a,y))$ is a constant for any $a$ such that there is $b$ with $(\tilde{F},\mbox{Frob})\models r(a,b)$, by additivity of $\mbox{dim}_{rk}$, we have $$\mbox{dim}_{rk}(r(x,y))=\mbox{dim}_{rk}(r(a,y))+\mbox{dim}_{rk}(\exists y~\! r(x,y)).$$ 
On the other hand, by subadditivity of $\pmb{\delta}_F$, we have
\begin{align*}
\pmb{\delta}_F(r(x,y))&\leq \pmb{\delta}_F(\exists y~\! r(x,y))+\max\{\pmb{\delta}_F(r(a,y)):a\in F,\text{ exists }b\in F, (F,\mbox{Frob})\models r(a,b)\}\\
&=\max\{\pmb{\delta}_F(r(a,y)):a\in F,\text{ exists }b\in F, (F,\mbox{Frob})\models r(a,b)\}
\end{align*}
Take $a'\in F$ such that there is $b'\in F$ with $(F,\mbox{Frob})\models r(a',b')$ and $\pmb{\delta}_F(r(a',y))$ reaches the maximal value. Then $\pmb{\delta}_F(r(x,y))\leq\pmb{\delta}_F(r(a',y))$. Since the other direction of the inequality always holds, we get $\pmb{\delta}_F(r(x,y))=\pmb{\delta}_F(r(a',y))$. Now applying Conjecture \ref{conj} to the type $r(x,y)$, we have
\begin{align*}
\mbox{dim}_{rk}(r(a',y))+\mbox{dim}_{rk}(\exists y~\! r(x,y))=\mbox{dim}_{rk}(r(x,y))\\
=\pmb{\delta}_F(r(x,y))=\pmb{\delta}_F(r(a',y))\leq \mbox{dim}_{rk}(r(a',y)).
\end{align*}
Therefore, $\mbox{dim}_{rk}(\exists y~\! r(x,y))=0$ as desired.

Recall that in the second case we have $\mbox{dim}_{rk}(r(a,y))<\mbox{dim}_{rk}(q(y))$ for any $a,b$ with $(\tilde{F},\mbox{Frob})\models r(a,b)$. By applying conjecture \ref{conj} to $q(y)$, we have
\begin{align*}
\mbox{dim}(q(y))&=\pmb{\delta}_F(q(y))\leq \pmb{\delta}_F(r(x,y))=\max\{\pmb{\delta}_F(ab/A):(F,\mbox{Frob})\models r(a,b)\}\\
&=\max\{\pmb{\delta}_F(b/A,a)+\pmb{\delta}_F(a/A):(F,\mbox{Frob})\models r(a,b)\}\\
&=\max\{\pmb{\delta}_F(b/A,a):(F,\mbox{Frob})\models r(a,b)\} \text{ (because }\exists y~\! r(x,y)\in\mbox{tp}(a/A))\\
&\leq \max\{\mbox{dim}_{rk}(b/A,a): (\tilde{F},\mbox{Frob})\models r(a,b)\}\\
&\leq \mbox{dim}_{rk}(r(a,y)) ~(\text{where there is }b \text{ with } (\tilde{F},\mbox{Frob})\models r(a,b))\\
&<\mbox{dim}_{rk}(q(y)),
\end{align*}
a contradiction.
\end{proof}

\begin{theorem}\label{thm02}
Let $r(x,y)$ be a quantifier-free complete type over $A\subseteq F$. Suppose Conjecture \ref{conj} holds. Then $$\pmb{\delta}_F(\exists y~\! r(x,y))=\mbox{dim}_{rk}(\exists y~\! r(x,y)).$$
\end{theorem}
\begin{proof}
Note that $\pmb{\delta}_F(\exists y~\! r(x,y))=\max\{\pmb{\delta}_F(a): (F,\mbox{Frob})\models\exists y~\! r(a,y)\}$ by $\omega$-saturation of $F$. Let $a\in F$ satisfying $ \exists y~\! r(x,y)$ such that $\pmb{\delta}_F(a)$ reaches the maximal value.
We can write $a=a_1a_2$ where $\pmb{\delta}_F(a/A)=\pmb{\delta}_F(a_1/A)=|a_1|$. We may assume $\exists y~\! r(x,y)$ to be $\exists y~\! r(x_1,x_2,y)$ and that $(F,\mbox{Frob})\models\exists y~\! r(a_1,a_2,y)$. We claim that $\pmb{\delta}_F(\exists y~\! r(a_1,x_2,y))=0$.
If not, we can extend the partial type $\exists y~\! r(a_1,x_2,y)$ to a complete type $p(x_2)$ over $A\cup\{a_1\}$ with $\pmb{\delta}_F(p(x_2))=\pmb{\delta}_F(\exists y~\! r(a_1,x_2,y))>0$. Let  $b\in F^{|x_2|}$ such that $(F,\mbox{Frob})\models p(b)$. Then we have $\pmb{\delta}_F(b/a_1,A)=\pmb{\delta}_F(p(x_2))>0$.
Therefore $$\pmb{\delta}_F(a_1,b/A)=\pmb{\delta}_F(a_1/A)+\pmb{\delta}_F(b/a_1,A)>\pmb{\delta}_F(a_1/A)=\pmb{\delta}_F(\exists y~\! r(x,y)).$$
However, as $\exists y~\! r(x,y)\in\mbox{tp}(a_1,b/A)$, we have $\pmb{\delta}_F(a_1,b/A)\leq \pmb{\delta}_F(\exists y~\! r(x,y))$, a contradiction.
By Lemma \ref{lem-conj}, $\mbox{dim}_{rk}(\exists y~\! r(a_1,x_2,y))=\pmb{\delta}_F(\exists y~\! r(a_1,x_2,y))=0$. Therefore, $\mbox{dim}_{rk}(a_2/a_1,A)=0$.
By Lemma \ref{lem01}, we have $$|a_1|=\pmb{\delta}_F(a_1/A)\leq\mbox{dim}_{rk}(a_1/A)\leq |a_1|.$$ By additivity of $\mbox{dim}_{rk}$, we  have  $$\mbox{dim}_{rk}(a/A)=\mbox{dim}_{rk}(a_2/a_1,A)+\mbox{dim}_{rk}(a_1/A)=\mbox{dim}_{rk}(a_1/A)=\pmb{\delta}_F(a/A).$$
Since $\mbox{dim}_{rk}(a/A)$ is determined by the quantifier-free type of $a$ over $A$, we have $$\mbox{dim}_{rk}(a/A)=\mbox{dim}_{rk}(\mbox{qftp}(a/A))=\mbox{dim}_{rk}(\exists y~\! r(x,y)),$$ and we conclude that $$\mbox{dim}_{rk}(\exists y~\! r(x,y))=\mbox{dim}_{rk}(a/A)=\pmb{\delta}_F(a/A)=\pmb{\delta}_F(\exists y~\! r(x,y)).\qedhere$$
\end{proof}

\begin{corollary}\label{lem-03}
Suppose Conjecture \ref{conj} holds. Let $\varphi(x):=\exists y\psi(x,y)$ be an $\mathcal{L}_\sigma$-existential formula defined over $A$. Then for any $a\in F^{|x|}$ with $(F,\mbox{Frob})\models \varphi(a)$, we have $\mbox{dim}_{rk}(a/A)\leq \pmb{\delta}_F(\varphi(x))$.
\end{corollary}
\begin{proof}

Suppose $a\in F^{|x|}$ and $(F,\mbox{Frob})\models \exists y \psi(a,y)$. Let $b\in F^{|y|}$ such that $(F,\mbox{Frob})\models \psi(a,b)$. Let $r(x,y)$ be the quantifier-free type of $a,b$ over $A$. Then $\pmb{\delta}_F(\exists y~\! r(x,y))\leq \pmb{\delta}_F(\exists y \psi(x,y))=\pmb{\delta}_F(\varphi(x))$. By Theorem \ref{thm02}, we have $\pmb{\delta}_F(\exists y~\! r(x,y))=\mbox{dim}_{rk}(\exists y~\! r(x,y))$. We conclude $$\mbox{dim}_{rk}(a/A)=\mbox{dim}_{rk}(\mbox{qftp}(a/A))=\mbox{dim}_{rk}(\exists y~\! r(x,y))=\pmb{\delta}_F(\exists y~\! r(x,y))\leq \pmb{\delta}_F(\varphi(x)).\qedhere$$

\end{proof}

In the following, we will discuss a positive evidence of Conjecture \ref{conj}. We will show that the conjecture is true for quantifier-free types in one variable.
To prove this, we will use the estimates on the number of solutions of formulas in ACFA, which is given in \cite{ryten2006acfa} based on Hrushovski's twisted Lang-Weil estimate.

\begin{definition}
Let $E$ be a model of ACFA and $A\subseteq E$. Let $\varphi(x)$ be a difference formula with parameters $A$. We define $$deg_\sigma(\varphi(x)):=\max\{deg_\sigma(a/A_\sigma):E\models \varphi(a)\}.$$
\end{definition}

\noindent \emph{Remark:} Given a formula $\varphi(x,y)$, seen as a family of definable sets parametrised by the variable $y$, by \cite[Section 7]{chatzidakis1999model}, the set $\{y:deg_{\sigma}(\varphi(x,y))=d\}$ is definable.

\begin{fact}\label{fact1}\cite[Theorem 1.1]{ryten2006acfa} and \cite[Theorem 2.1.1]{ryten2007model}
Let $K_q:=(\tilde{\mathbb{F}}_p,\mbox{Frob}_q:x\mapsto x^q)$ where $q$ is a power of the prime number $p$. Let $\varphi(x,y)$ be a formula in the language of difference rings, with $x=(x_1,\ldots,x_n)$ and $y=(y_1,\ldots,y_m)$. Then there is a positive constant $C$ and a finite set $D$ of pairs $(d,\mu)$ with $D\subseteq \mathbb{Z}$ and $\mu\in \mathbb{Q}^+$, such that in each field $K_q$ and each $y_0\in K_q^m$, one of the following happens:
\begin{enumerate}
\item
There is some $(d,\mu)\in D$ such that $deg_\sigma(\varphi(x,y_0))=d$, and we have the estimate
$$||\varphi(K_q^n,y_0)|-\mu q^d|\leq Cq^{d-\frac{1}{2}}.$$ 
\item
$deg_\sigma(\varphi(x,y_0))=\infty$ and $|\varphi(K_q^n,y_0)|=\infty.$
\end{enumerate}
\end{fact}

\begin{lemma}\label{lem-dimsigma}
Let $\varphi(x,b)$ (with $b\in F^n$ for some $n\geq 1$) be an $\mathcal{L}_\sigma$-quantifier-free formula such that $|x|=1$ and $\pmb{\delta}_F(\varphi(x,b))=0$. Then $deg_{\sigma}(\varphi(x,b))<\infty$. In particular, $\pmb{\delta}_F(\varphi(x,b))=\mbox{dim}_{rk}(\varphi(x,b))$.
\end{lemma}
\begin{proof}
Suppose $b=(b_p)_{p\in\mathbb{P}}/\mathcal{U}$. Let $\varphi_p(x,y)$ be defined as in Definition \ref{def-phip}.  As $\varphi_p(x,y)$ is quantifier-free and $|x|=1$, by Fact \ref{fact0}, there is a constant $C_p$ such that for all $\mathbb{F}_{p^k}$ and $d\in (\mathbb{F}_{p^k})^{|y|}$, either $|\varphi_p(\mathbb{F}_{p^k},d)|<C_p$ or $p^k-|\varphi_p(\mathbb{F}_{p^k},d)|<C_p\cdot p^{\frac{k}{2}}$. By our construction $k_p>C_p$. As $\pmb{\delta}_F(\varphi(x,b))=0$, there is some $V$ in the ultrafilter $\mathcal{U}$ which has the following property: 
$|\varphi_p(\mathbb{F}_{p^{k_p}},b_p)|<C_p$ for all $p\in V$. 
Note that $\varphi_p(\tilde{\mathbb{F}}_p,b_p)$ is either finite or co-finite. We may assume $\min\{|\varphi_p(\tilde{\mathbb{F}}_p,b_p)|,|\neg\varphi_p(\tilde{\mathbb{F}}_p,b_p)|\}<C_p$ and $C_p>2$. If $|\neg\varphi_p(\tilde{\mathbb{F}}_p,b_p)|<C_p$, then $|\varphi(\mathbb{F}_{p^{k_p}},b_p)|\geq p^{k_p}-C_p>p^{C_p}-C_p>C_p$. Therefore, $|\varphi_p(\tilde{\mathbb{F}}_p,b_p)|< C_p$.

Note that $\varphi_p(\tilde{\mathbb{F}}_p,b_p)$ is exactly the set $\varphi(K_p,b_p)$, where $K_p=(\tilde{\mathbb{F}}_p,\mbox{Frob}_p)$. Then, since $|\varphi(K_p,b_p)|=|\varphi_p(\tilde{\mathbb{F}}_p,b_p)|<\infty$ for each $p\in V$, we get by Fact \ref{fact1} a finite set $D\subseteq\mathbb{N}\times \mathbb{Q}^+$ such that for any $p\in V$, there is some $(d,\mu)\in D$ and the following holds: $$||\varphi(K_p,b_p)|-\mu p^d|\leq C p^{d-\frac{1}{2}}.$$ Therefore, there is some $J\in\mathcal{U}$, $J\subseteq V$ and one particular pair $(d,\mu)\in D$ such that for any $p\in J$, we have $||\varphi(K_p,b_p)|-\mu p^d|\leq C p^{d-\frac{1}{2}}$. By Fact \ref{fact1} we know that $deg_{\sigma}(\varphi(x,b_p))=d$ for any $b_p\in J$. By the previous remark, we know there is some formula $\varphi_d(y)$, such that $\varphi_d(y)$ holds in a difference field if and only if $deg_{\sigma}(\varphi(x,y))=d$. Therefore, $\varphi_d(b_p)$ holds in each $K_p$ with $p\in J$, hence $\varphi_d(b)$ holds in $(\tilde{F},\mbox{Frob})$. We conclude that $$deg_{\sigma}(a/b)\leq deg_{\sigma}(\varphi(x,b))=d.\qedhere$$
\end{proof}

If we can establish the connection between $\mbox{dim}_{rk}$ and $\pmb{\delta}_F$, then it will help us to explore more properties of $(F,\mbox{Frob})\in\mathcal{S}$. Let $X$ be a definable object in $(F,\mbox{Frob})$. If we have the control of $\mbox{dim}_{rk}$ of $X$, then we work in $(\tilde{F},\mbox{Frob})$. As it is a model of ACFA, we can use all the model-theoretic tools there. Finally, we transfer the results from $(\tilde{F},\mbox{Frob})$ back to $(F,\mbox{Frob})$. In the following, we will give an example which is about understanding definable subgroups of algebraic groups.

\begin{fact}\label{fact3}\cite[Section 6.5]{zoeND}
Let $(k,\sigma)$ be a model of ACFA. Let $G$ be a definable subgroup of some algebraic group $H(k)$. Let $\mbox{acl}_{\sigma}$ denote the algebraic closure in ACFA. Suppose $G$ is definable over $E=\mbox{acl}_{\sigma}(E)$. Then $G$ is contained in a group $\tilde{G}$ which is quantifier-free definable over $E$ and has the same SU-rank as $G$.
\end{fact}

\begin{lemma}\label{lem-polynom}
Let $(F,\mbox{Frob})\in\mathcal{S}$, $a\in F^n$ and $A\subseteq F$. Suppose $\mbox{dim}_{rk}(a/A)=k$. Then there is a finite set $\{P_1(x),\ldots,P_m(x)\}$ of difference polynomials with parameters in $A$ such that $(F,\mbox{Frob})\models \bigwedge_{i\leq m}P_i(a)=0$ and $\mbox{dim}_{rk}(\bigwedge_{i\leq m }P_i(x)=0)=k$.
\end{lemma}
\begin{proof}
We may write $a$ into two parts $a_1$ and $a_2$ where $\mbox{dim}_{rk}(a_1/A)=|a_1|=k$, and $\mbox{dim}_{rk}(a_2/Aa_1)=0$.  Let $(Aa_1)_\sigma$ be the difference field generated by $A\cup\{a_1\}$. Suppose $a_2:=a_2^1\cdots a_2^m$ with each $|a_2^i|=1$. Since $\mbox{dim}_{rk}(a_2^i/Aa_1)=0$ for each $i\leq m$, we get $deg_{\sigma}(a_2^i/(Aa_1)_{\sigma})<\infty$. Therefore, there is a difference polynomial $P_i(y_i,b_i)$ with $b_i\subseteq (Aa_1)_\sigma$ such that $a_2^i$ vanishes on it. Write $b_i=f_i(a_1)$ where $f_i$ is a difference polynomial with parameters in $A$. We should rearrange the order of variables such that $x_0,\ldots,x_{|a|-1}$ corresponds to the order of $a$. Suppose $a_1=a^{\ell_1}\cdots a^{\ell_{|a_1|}}$ and $a_2=a^{t_1}\cdots a^{t_{|a_2|}}$ where $a^j$ is the $j^{th}$ component of the tuple $a$. Now it is easy to see that $a$ satisfies the formula $$\varphi(x):=\bigwedge_{i\leq m}P_i(x_{t_i},f_i(x_{\ell_1},\ldots,x_{\ell_{|a_1|}}))=0,$$ and $\mbox{dim}_{rk}(\varphi(x))=k$.
\end{proof}

\begin{theorem}\label{cor-defgp} 
Let $(F,\mbox{Frob})\in\mathcal{S}$. Suppose $G$ is a definable subgroup of some algebraic group $H(F)\subseteq F^n$, both defined over a finite set $A\subseteq F$. If for any $g\in G$ we have $\mbox{dim}_{rk}(g/A)\leq \pmb{\delta}_F(G)$, then there is a quantifier-free definable group $\bar{G}\geq G$ (defined with parameters in $F$ which possibly extends $A$), such that $\pmb{\delta}_F(\bar{G})=\pmb{\delta}_F(G)$.

In particular, if Conjecture \ref{conj} holds and $G$ is a definable subgroup of an algebraic group and $G$ is defined by an existential formula, then there is a quantifier-free definable group $\bar{G}\geq G$ such that $\pmb{\delta}_F(\bar{G})=\pmb{\delta}_F(G)$.
\end{theorem}
\begin{proof}
Suppose $G$ is defined by the formula $\varphi_G$. Let $k:=\pmb{\delta}_F(G)$.

Let $\Pi_A$ denote the set of difference polynomials in $n$-variables with coefficients in $A$.

By Lemma \ref{lem-polynom}, for any element $a\in G$, there are some $\{P_{a,i}(x):1\leq i\leq m_a\}\subset\Pi_A$ such that $(F,\mbox{Frob})\models \bigwedge_{i\leq m_a}P_{a,i}(a)=0$ and $\mbox{dim}_{rk}(\bigwedge_{i\leq m_a}P_{a,i}(x)=0)=\mbox{dim}_{rk}(a/A)$. By assumption, $\mbox{dim}_{rk}(a/A)\leq \pmb{\delta}_F(G)=k$.
Therefore, $\varphi_G(x)$ is covered by the collection of formulas $\{\bigwedge_{i\leq m_a}P_{a,i}(x)=0:a\in G\}$. Since $[\varphi_G]$ is closed in the compact space $S_n(F)$, we have by compactness, there is some finite set $a_0,\ldots,a_\ell$ such that $\varphi_G(x)\models \bigvee_{j\leq \ell}\left(\bigwedge_{i\leq m_{a_j}} P_{a_j,i}(x)=0\right)$. Let $\Phi(x):=\bigvee_{j\leq \ell}\left(\bigwedge_{i\leq m_{a_j}} P_{a_j,i}(x)=0\right)$. As $\mbox{dim}_{rk}(\bigwedge_{i\leq m_{a_j}}P_{{a_j},i}(x)=0)\leq k$ for each $j\leq \ell$, we get $\mbox{dim}_{rk}(\Phi(x))\leq k$.

Write $\Phi(x)$ into the conjunctive normal form $\displaystyle{\bigwedge_{u\leq N}\bigvee_{v\leq M_u}(P_{u,v}(x)=0)}$ for some natural numbers $N,M_u$, and each $P_{u,v}(x)\in\{P_{a_j,i}(x):j\leq \ell,i\leq m_{a_j}\}$. Hence, for each $u\leq N$, we have $\varphi_G(x)\models (\prod_{v\leq M_u}P_{u,v}(x))=0.$ 

Let $G_{\tilde{F}}$ be the $\sigma$-Zariski closure of $G$ in $H(\tilde{F})$, that is, if we define $I_{\tilde{F}}(G)=\{p\in \tilde{F}[x]_{\sigma}:p(g)=0 \text{ for all }g\in G\},$ then $$G_{\tilde{F}}:=\{h\in H(\tilde{F}):p(h)=0\text{ for all  }p\in I_{\tilde{F}}(G)\}.$$ As prime $\sigma$-ideals are finitely generated, $G_{\tilde{F}}$ is quantifier-free definable. Note that $\prod_{v\leq M_u}P_{u,v}(x)\in I_{\tilde{F}}(G)$ for each $u\leq N$. Since $$\mbox{dim}_{rk}\left(\bigwedge_{u\leq N}\left(\prod_{v\leq M_u}P_{u,v}(x)~\right)=0~\right)=\mbox{dim}_{rk}\left(\bigvee_{j\leq \ell}\bigwedge_{i\leq m_{a_j}}P_{{a_j},i}(x)=0~\right)\leq k,$$ we get $\mbox{dim}_{rk}(G_{\tilde{F}})\leq k$. 

Take an automorphism $\alpha$ of $(\tilde{F},\mbox{Frob})$ fixing $F$. Then $G=\alpha(G)\subseteq \alpha(G_{\tilde{F}})$. As $\alpha(G_{\tilde{F}})$ is also closed under the $\sigma$-Zariski topology in $(\tilde{F},\mbox{Frob})$, we get $G_{\tilde{F}}\subseteq \alpha(G_{\tilde{F}})$ which implies $G_{\tilde{F}}= \alpha(G_{\tilde{F}})$. Therefore, $G_{\tilde{F}}$ is invariant under automorphisms fixing $F$, hence it is definable over $F$. 
Let $E=\mbox{acl}_{\sigma}(F)= F^{alg}$, then by Fact \ref{fact3}, there is $G_E$ which contains $G_{\tilde{F}}$, has the same SU-rank as $G_E$ and is quantifier-free definable over $E$. In fact, $G_E$ is the smallest closed set containing $G_{\tilde{F}}$ in the $\sigma$-Zariski topology in $(F^{alg},\mbox{Frob}\upharpoonright_{F^{alg}})$.

Suppose $G_{E}$ is defined by $$\bigwedge_{0\leq j\leq \ell'}P'_j(x,\sigma(x),\ldots,\sigma^m(x),c_j)=0,$$ where $P'_j$ are polynomials in the language of rings and $c_j\subseteq F^{alg}$. For any $0\leq j\leq \ell'$, let $\{c_j^0,\ldots ,c_j^{N_j}\}\subseteq (F^{alg})^{|c_j|}$ be the set of all field conjugates of $c_j$ over $F$. Note that for any $g\in G$ we have $g,\sigma(g),\ldots,\sigma^m(g)\subseteq F$. Hence, $P'_j(g,\sigma(g),\ldots,\sigma^m(g),c_j)=0$ if and only if $P'_j(g,\sigma(g),\ldots,\sigma^m(g),c_j^i)=0$ for any $g\in G$ and $0\leq i\leq N_j$.

Let $B_j$ be the set in $H(\tilde{F})$ vanishing on $\{P'_j(x,\sigma(x),\ldots,\sigma^m(x),c_j^i):0\leq i\leq N_j\}$. Then from the above argument, we know $B_j\supseteq G$. As $B_j$ is closed under the $\sigma$-Zariski topology in $(\tilde{F},\mbox{Frob})$, we get $B_j\supseteq G_{\tilde{F}}$. Similarly, by $B_j$ being closed under the $\sigma$-Zariski topology in $(F^{alg},\mbox{Frob}\upharpoonright_{F^{alg}})$, we get $B_j\supseteq G_E$.

Now consider the formula $$\bigwedge_{0\leq j\leq \ell'}~\bigwedge_{0\leq i\leq N_j}P'_j(x,\sigma(x),\ldots,\sigma^m(x),c^i_j)=0.$$ It defines $\bigcap_{j\leq \ell'}B_j$. As before, we know that $\bigcap_{j\leq\ell'}B_j\supseteq G_E$. Clearly, we also have $\bigcap_{j\leq\ell'}B_j\subseteq G_E$. Hence, the formula above also defines $G_E$ in $H(\tilde{F})$. Now we show that $G_E$ can be made quantifier-free definable over $F$.

Fix $0\leq j\leq\ell'$ and consider the formula $$\bigwedge_{0\leq i\leq N_j}P'_j(x,x_1,\ldots,x_m,c^i_j)=0,$$ where $x_1,\ldots,x_m$ are distinct tuples of variables all have the same length as $x$. For $1\leq k\leq N_j+1$, let $e_k(t_0,\ldots,t_{N_j})$ be the $k$-elementary symmetric polynomials in $N_j+1$-variables, i.e.\ $$e_k(t_0,\ldots,t_{N_j}):=\sum_{0\leq i_1<\cdots<i_k\leq N_j}t_{i_1}\cdots t_{i_k}.$$ Then we have $\bigwedge_{0\leq i\leq N_j}P'_j(x,x_1,\ldots,x_m,c^i_j)=0$ if and only if $$\bigwedge_{1\leq k\leq N_j+1}e_k(P'_j(x,x_1,\ldots,x_m,c_j^0),\ldots,P'_j(x,x_1,\ldots,x_m,c_j^{N_j}))=0.$$ For each $1\leq k\leq N_j+1$, as $\{c^i_j:0\leq j\leq N_j\}$ is the set of all field conjugates of $c_j$ in $F^{alg}$ over $F$ and that $e_k$ is symmetric, we get $$Q^k_j(x,\ldots,x_m,b_j^k):=e_k(P'_j(x,x_1,\ldots,x_m,c_j^0),\ldots,P'_j(x,x_1,\ldots,x_m,c_j^{N_j}))$$ is invariant under field automorphisms in $\mbox{Gal}(F^{alg}/F)$. Therefore, since $F$ is a pseudofinite field, $F$ is perfect and we have $b_j^k\subseteq F$ for all $1\leq j\leq\ell'$ and $1\leq k\leq N_j+1$.

Let $\varphi_H(x)$ be the quantifier-free formula with parameters in $A$ that defines the algebraic group $H$. Now consider $$\psi(x):=\varphi_H(x)\land \left(\bigwedge_{0\leq j\leq\ell'}~\bigwedge_{1\leq k\leq N_j+1}Q^k_j(x,\sigma(x),\ldots,\sigma^m(x),b_j^k)=0\right).$$ It is easy to see that $\psi(x)$ defines $G_E$ in $(\tilde{F},\mbox{Frob})$. Note that $\psi(x)$ is quantifier-free and defined over $F$, so we can consider $\bar{G}:=\{g\in F^t:(F,\mbox{Frob})\models \psi(g)\}$. Since $H(F)$ is an algebraic group and $F$ is definably closed in $\tilde{F}$ in the language of rings, $\bar{G}$ is a quantifier-free definable group in $(F,\mbox{Frob})$ and contains $G$. Note that $\mbox{dim}_{rk}(G_E)=\mbox{dim}_{rk}(G_{\tilde{F}})\leq k$. Hence, $\pmb{\delta}_F(\bar{G})\leq \mbox{dim}_{rk}(\psi(x))=\mbox{dim}_{rk}(G_E)\leq k$. On the other hand, since $\bar{G}\supseteq G$ and $\pmb{\delta}_F(G)=k$, we get $\pmb{\delta}_F(\bar{G})\geq k$. Therefore, $\pmb{\delta}_F(\bar{G})=\pmb{\delta}_F(G)=k$, which concludes the proof of Theorem \ref{cor-defgp}.
\end{proof}

\section{Wildness of $\mathcal{S}$}\label{sec4}
This section will be some discussions about negative model-theoretic properties of the class $\mathcal{S}$ defined in Section \ref{sec2}. We will first investigate whether this family $\mathcal{S}$ is \emph{tame} in terms of the properties in Shelah's classification theory \cite{shelah1990classification}. It turns out that the answer is negative. As we have mentioned before, we will show that if a structure expands a pseudofinite field with a ``logarithmically small'' definable subset, then all the internal subsets of this definable set will be uniformly definable.\footnote{This result is known among experts. As we could not find a proof in the literature, we include it here for completeness.} Therefore, theories of structures in $\mathcal{S}$ have TP2 and the strict order property and is not decidable. We proceed by an example in $\mathcal{S}$ where the model-theoretic algebraic closure does not coincide with the algebraic closure in the sense of difference algebra. We conclude with some general remarks and questions.

\subsection{Non-tameness}
In this subsection we will show that the theory of any member of $\mathcal{S}$ has TP2 and the strict order property and is not decidable.

The proof is based on the result that the theory of pseudofinite fields has the independence property in \cite{duret}. The strategy is to modify Duret's proof to show that when an internal set is very small compared to the size of the field, then every internal subset of it can also be coded uniformly.

\begin{fact} (\cite[Proposition 4.3]{duret}) \label{duret4.3}
Let $k$ be a field and $p$ a prime different from $\mbox{char}(k)$ such that $k$ contains a $p^{th}$-root of unity.  Let $\tilde{k}$ be the algebraic closure of $k$. Suppose $f_i\in k[Y_1,\ldots,Y_m]$ and $F_i=X^p-f_i\in k[Y_1,\ldots,Y_m,X]$ for $1\leq i\leq n$. If there exist $g_i,h_i\in \tilde{k}[Y_1,\ldots,Y_m]$ and $q_i\in\mathbb{N}$ such that:
\begin{itemize}
\item
for all $i,~ f_i=g_i^{q_i} h_i$;
\item
for all $i,~g_i\text{ is prime in }\tilde{k}[Y_1,\ldots,Y_m]$
\item
for all $i \neq j,~g_i\neq g_j$
\item
for all $i\text{ and } j,~g_i\text{ does not divide }h_j$
\item
for all $i,~p\text{ does not divide }q_i.$ 
\end{itemize}

Then the ideal $J$ in $k[Y_1,\ldots,Y_m,X_1,\ldots,X_n]$ generated by $\{F_i(X_i):~1\leq i\leq n\}$ is absolutely prime, and does not contain any non-zero element in $k[Y_1,\ldots,Y_m]$.
\end{fact}

\begin{fact}(\cite[Theorem 7.1]{cafure})\label{lem-counting}
Let $V\subseteq (\tilde{\mathbb{F}}_q)^n$ be an absolutely irreducible $\mathbb{F}_q$-variety of dimension $r>0$ and degree $\ell$. If $q>2(r+1)\ell^2$, then the following estimate holds:
$$||(V\cap (\mathbb{F}_q)^n)|-q^r|\leq(\ell-1)(\ell-2)q^{r-\frac{1}{2}}+5\ell^{\frac{13}{3}}q^{r-1}.$$
\end{fact}

\begin{theorem}\label{thm-sop}
Let $F=\prod_{i\in I}\mathbb{F}_{q_i}/\mathcal{U}$ be a pseudofinite field and $A=\prod_{i\in I}A_i/\mathcal{U}$ an infinite internal subset of $F$. Suppose there is a positive constant $C$ such that $\{i\in I:|A_i|\leq C\log_2q_i\}\in\mathcal{U}$. Then all internal subsets of $A$ are uniformly definable.
\end{theorem}
\begin{proof}
Consider the finite algebraic extension $F'$ of $F$ of degree $14\lceil C\rceil$. As $F$ is pseudofinite, there is only one such extension and is definable. To see the definability, suppose $F'=F(\alpha)$. Let $f$ be the minimal polynomial of $\alpha$ over $F$. Then we can define $F'$ as the $14\lceil C\rceil$-dimensional vector space over $F$ with multiplication defined according to the minimal polynomial $f$.

We distinguish two cases according to $p_i:=\mbox{char}(\mathbb{F}_{q_i})$. First, let us suppose $p_i\neq 2$ and $q_i=p_i^{n_i}$. Since $x^{p_i^{14\lceil C\rceil n_i}-1}=\mathsf{1}$ for all $x\in \mathbb{F}_{p_i^{14\lceil C\rceil n_i}}$, the square root of unity exists in $\mathbb{F}_{p_i^{14\lceil C\rceil n_i}}$. As the multiplicative group of $\mathbb{F}_{p_i^{14\lceil C\rceil n_i}}$ is cyclic, take $\alpha_i\in \mathbb{F}_{p_i^{14\lceil C\rceil n_i}}$ a generator, then $\alpha_i$ is not a square in  $\mathbb{F}_{p_i^{14\lceil C\rceil n_i}}$.

\begin{claim}
Let $\varphi(y,u)$ be the formula $\exists x(x^2=y+u).$ Then for all $i\in I$ with $p_i\neq 2$ and for all $E_i\subseteq A_i$, there is $y_i\in F_{p_i^{14\lceil C\rceil n_i}}$ such that $$E_i=\varphi(y_i,\mathbb{F}_{p_i^{14\lceil C\rceil n_i}})\cap A_i.$$
\end{claim}
\begin{proof}
Let $i\in I$ with $p_i\neq 2$, $E_i\subseteq A_i$ and $t_i:=|A_i|\leq Cn_i\log_2 p_i$. Let $J$ be the ideal in $\mathbb{F}_{p_i^{14\lceil C\rceil n_i}}[X_1,\ldots,X_{t_i},Y]$ generated by $$\{X_j^2-(Y+c_j):c_j\in E_i\}\cup\{X_j^2-\alpha_i(Y+d_j):d_j\in A_i\setminus E_i\},$$ where $\alpha_i$ is a generator of $\mathbb{F}_{p_i^{14\lceil C\rceil n_i}}^{\times}$ as defined before. Let $V(J)$ be the corresponding $\mathbb{F}_{p_i^{14\lceil C\rceil n_i}}$-variety. Then $V(J)$ is absolutely irreducible by Fact \ref{duret4.3},

Suppose $V(J)\cap (\mathbb{F}_{p_i^{14\lceil C\rceil n_i}})^{t_i+1}\neq \emptyset$. Let $(x_1,\ldots,x_{t_i},y_i)$ be a solution. Then clearly $E_i\subseteq \varphi(y_i,\mathbb{F}_{p_i^{14\lceil C\rceil n_i}})$. On the other hand, if there is $d\in A_i\setminus E_i$, such that $\varphi(y_i,d)$. Then there are $x_j,x\in \mathbb{F}_{p_i^{14\lceil C\rceil n_i}}$ such that:
$$\begin{aligned}
x_j^2=\alpha_i(y_i+d);\\
x^2=y_i+d;\\
y_i-d\neq 0,
\end{aligned}$$
where the last inequality follows from Fact \ref{duret4.3}, as $Y-d\not\in J$. Hence, $\alpha_i=\left(\frac{x_j}{x}\right)^2$, contradicting that $\alpha_i$ is not a square root. Therefore, $E_i=\varphi(y_i,\mathbb{F}_{p_i^{14\lceil C\rceil n_i}})\cap A_i$. 

So we only need to show $V(J)\cap\mathbb{F}_{p_i^{14\lceil C\rceil n_i}}\neq\emptyset$.

Let $|A_i|=t_i\leq Cn_i\log_2p_i$. We calculate the dimension and the degree of $V(J)$. It is clear that the dimension of $V(J)$ is $1$, as all $X_j$ are algebraic over $Y$. Let $c_1,\ldots,c_{t_i}$ be a list of all elements in $A_i$, and for $1\leq j\leq t_i$, let $V_j$ be the variety defined by either the set of solutions of $X_j^2-(Y+c_j)$ if $c_j\in E_i$, or $X_j^2-\alpha_i(Y+c_j)$ if $c_j\not\in E_i$. Then $V(J)=\bigcap_{1\leq j\leq t_i}V_j$ and each $V_j$ has degree $2$. Therefore, by the B\'{e}zout inequality, the degree of $V(J)$ is less than or equal to $2^{t_i}$.

Suppose, towards a contradiction, that $V(J)\cap (\mathbb{F}_{p_i^{14\lceil C\rceil n_i}})^{t_i+1}=\emptyset$. Then by Fact \ref{lem-counting}, 
\begin{align*}
p_i^{14\lceil C\rceil n_i} & \leq (2^{t_i}-1)(2^{t_i}-2)p_i^{7\lceil C \rceil n_i}+5\times 2^{\frac{13}{3}t_i}\\
& \leq (p_i^{Cn_i}-1)(p_i^{Cn_i}-2)p_i^{7\lceil C\rceil n_i}+5\times p_i^{\frac{13}{3}Cn_i}\\
& < p_i^{2Cn_i}p_i^{7\lceil C\rceil n_i}+p_i^{8Cn_i}= p_i^{9\lceil C\rceil n_i}+p_i^{8Cn_i}\\
& <p_i^{14\lceil C\rceil n_i},
\end{align*}
contradiction.\qedhere
\end{proof}

The case $\mbox{char}(q_i)=2$ is similar. Suppose $q_i=2^{n_i}$. Since $3$ divides $2^{14\lceil C\rceil n_i}-1$ for each $i$, there exists $x\in \mathbb{F}_{2^{14\lceil C\rceil n_i}}$ such that $x^3=1$. Take $\beta_i$ to be the generator of the multiplicative group of $\mathbb{F}_{2^{14\lceil C\rceil n_i}}$. Then there is no $y\in \mathbb{F}_{2^{14\lceil C\rceil n_i}}$ such that $y^3=\beta_i$.
\begin{claim}
Let $\psi(y,u)$ be the formula $\exists x(x^3=y+u).$ Then for all $i\in I$ and $E_i\subseteq A_i$, there is $y_i\in \mathbb{F}_{2^{14\lceil C\rceil n_i}}$ such that $E_i=\psi(y_i,\mathbb{F}_{2^{14\lceil C\rceil n_i}})\cap A_i$.
\end{claim}
\begin{proof}
Fix some $i$ and $E_i\subseteq A_i$. Let $J$ be the ideal in $\mathbb{F}_{2^{14\lceil C\rceil n_i}}[X_1,\ldots,X_{t_i},Y]$ generated by $$\{X_j^3-(Y+c_j):c_j\in E_i\}\cup\{X_j^3-\beta_i(Y+d_j):d_j\in A_i\setminus E_i\}.$$

As in the previous argument, the variety $V(J)$ is absolutely irreducible of dimension 1 and of degree less than or equal to $3^{t_i}$. To prove the claim, we only need to show that $V(J)\cap (\mathbb{F}_{2^{14\lceil C\rceil n_i}})^{t_i+1}\neq\emptyset$. Suppose not, then by Fact \ref{lem-counting}, $$2^{14\lceil C\rceil n_i}\leq (3^{t_i}-1)(3^{t_i}-2)2^{7\lceil C\rceil n_i}+5\times 3^{\frac{13}{3} t_i}\leq 3^{2 C n_i}2^{7\lceil C\rceil n_i}+3^{7 C n_i}
<2^{14\lceil C\rceil n_i},$$ contradiction.
\end{proof}
Let $A=\prod_{i\in I}A_i/\mathcal{U}$. Assume $A$ is defined by $\chi(x)$. Define $\phi(x,y):=\psi(y,x)\land \chi (x)$ if the characteristic of $F'$ is 2, and $\phi(x,y):=\varphi(y,x)\land \chi(x)$ otherwise. Let $E=\prod_{i\in I}E_i/\mathcal{U}\subseteq A$ be any internal subset. By the previous two claims, there is $y_E\in F'$ such that $E=\phi(F',y_E)$ in $F'$. Remember that we regard $F'$ as $14\lceil C\rceil$-dimensional vector space over  $F$ and $A\subseteq F$. So as $F'$ is definable in $F$, let $\phi'(\bar{x},\bar{y})$ be the corresponding translation of $\phi(x,y)$ in $F$ and put $\theta(x,\bar{y}):=\phi'(x,0,\ldots,0,\bar{y})$. We see that $\theta(x,\bar{y})$ codes uniformly all internal subsets of $A$.
\end{proof}
\noindent \emph{Remark:} 
\begin{itemize}
\item
From the proof we know that if $\mbox{char}(F)\neq 2$ and $q_i\geq 2^{14 |A_i|}$ for all large enough $i$, then we can take $\theta(x,\bar{y}):=\exists z^2(z^2=x+y)\land \chi(x)$ where $x,y$ are single variables and $\chi(x)$ is the formula defining $A$.
\item
The above proof of Theorem \ref{thm-sop} is purely algebraic. However, it is possible to use the Paley graphs $(P_q,R)$ constructed from $\mathbb{F}_q$ and the Bollob\'{a}s-Thomason inequalities to give a combinatoric and more neat proof when $q\equiv 1~(\mbox{mod}~ 4)$.\footnote{We would like to thank the referee to point out this observation. In fact, the Bollob\'{a}s-Thomason inequality will give a better bound than the bound we use for the Lang-Weil estimate in Fact \ref{lem-counting}. But the author has not yet found the equivalent Bollob\'{a}s-Thomason inequality in the characteristic 2 case.} The idea is that suppose we have a small subset $A\subseteq\mathbb{F}_q$ with $|A|=m$ and $E\subseteq A$. Let $V(E,A\setminus E)$ be set of vertices in $\mathbb{F}_q$ not in $A$ which connect to everything in $E$ and nothing in $A\setminus E$. Then the Bollob\'{a}s-Thomason inequality will give $$\left||V(E,A\setminus E)|-2^{-m}q\right|\leq \frac{1}{2}\left(m-2+2^{-m+1}\right)q^{\frac{1}{2}}+\frac{m}{2}.$$ Hence, when $q>>2^m$, then $V(E,A\setminus E)\neq \emptyset$. And any element in $V(E,A\setminus E)$ will code the subset $E$ inside $A$, and the coding is uniform by the formula $\varphi(x,y):=x\in A\land xRy$.
\end{itemize}

\begin{corollary}Let $F=\prod_{i\in I}\mathbb{F}_{q_i}/\mathcal{U}$ be a pseudofinite field and $B=\prod_{i\in I}B_i/\mathcal{U}$ an infinite internal subset of $F$. Suppose there is a positive constant $C$ such that $\{i\in I:|B_i|\leq C\log_2 q_i\}\in\mathcal{U}$. Then $(F,B)$ interprets the structure $N=\prod_{i\in I}(N_i,+,\times)/\mathcal{U}$, where $N_i=\{j\in\mathbb{N}:0\leq j\leq m_i\}$ for some $m_i\in\mathbb{N}$, and $+,\times$ are the addition and multiplication truncated on $N_i$ respectively.
\end{corollary}
\begin{proof}
For each $i\in I$, pick $Y_i\subseteq B_i$ such that $|B_i|^{\frac{1}{4}}\leq |Y_i|\leq |B_i|^\frac{1}{3}$. Let $Y=\prod_{i\in I}Y_i/\mathcal{U}$. By Theorem \ref{thm-sop}, $Y$ is definable and all subsets of $Y_i$ are uniformly definable by some $\psi_1(y,u)$. For each $i\in I$, consider the set $W_i:=\left\lbrace \dfrac{y_1-y_2}{y_3-y_4}:y_1,y_2,y_3,y_4\in Y_i,y_3\neq y_4\right\rbrace.$ The set $W_i$ has size at most $|Y_i|^4<<|\mathbb{F}_{q_i}|$. Take any $ a\not\in W_i\cup\{\mathsf{0}\}$. Then the set $T_i:=\{y_1+ay_2:y_1,y_2\in Y_i\}$ is in definable bijection with $Y_i\times Y_i$ and of size less than $\log_2q_i$. By Theorem \ref{thm-sop}, all subsets of $T_i$, hence of $Y_i\times Y_i$, are uniformly definable by some $\psi_2(y,u)$. Similarly, we can show that all subsets of $Y_i\times Y_i\times Y_i$ are uniformly definable by some $\psi_3(y,u)$. 

For $a\in \mathbb{F}_{q_i}$, we write $S^1_a\subseteq Y_i$ for the set $\psi_1(a,\mathbb{F}_{q_i})$  and $S^2_a\subseteq Y_i\times Y_i$, $S^3_a\subseteq Y_i\times Y_i\times Y_i$ for $\psi_2(a,\mathbb{F}_{q_i}),\psi_3(a,\mathbb{F}_{q_i})$ respectively. 

Now define a relation $R_+\subseteq (\mathbb{F}_{q_i})^3$ by: $R_+(a,b,c)$ if there exist $g\in \mathbb{F}_{q_i}$ and $y\neq y'\in Y_i$ such that
\begin{itemize}
\item
either $S^3_g$ is the graph of a bijective function from $(S^1_a\times\{y\})\cup (S^1_b\times\{y'\})$ to $S^1_c$;
\item
or $S^1_c=Y_i$ and $S^3_g$ is the graph of a surjective function from $(S^1_a\times\{y\})\cup (S^1_b\times\{y'\})$ to $Y_i$;
\end{itemize}

Similarly, we define $R_\times\subseteq (\mathbb{F}_{q_i})^3$ by: $R_\times(a,b,c)$ if there exists $g\in \mathbb{F}_{q_i}$ such that 
\begin{itemize}
\item
either $S^3_g$ is the graph of a bijective function from $S^1_a\times S^1_b$ to $S^1_c$;
\item
or $S^1_c=Y_i$ and $S^3_g$ is the graph of a surjective function from $S^1_a\times S^1_b$ to $Y_i$;
\end{itemize}

We also define an equivalence relation $E\subseteq (\mathbb{F}_{q_i})^2$ by: $E(a,b)$ if and only if there exists $g\in \mathbb{F}_{q_i}$ such that $S^2_g$ is the graph of a bijective function from $S^1_a$ to $S^1_b$.

It is easy to see then that $R^+,R^\times$ respect the equivalence relation $E$ and $$(|Y_i|,+,\times)\simeq ((\mathbb{F}_{q_i})^2/E,R^+/E,R^\times/E).$$
\end{proof}


\begin{corollary}\label{cor-interval}
Let $(F,\mbox{Frob})\in\mathcal{S}$ and $T:=Th(F,\mbox{Frob})$. Then $T$ has the strict order property and TP2. Moreover, $T$ is not decidable.
\end{corollary}
\begin{proof}
As the fixed field $\mbox{Fix}(F):=\{x\in F:\sigma(x)=x\}$ is definable and satisfies the condition in Theorem \ref{thm-sop}, every internal subset of $\mbox{Fix}(F)$ can be coded uniformly by some formula $\varphi(x,t)$. In particular, it will code some infinite strictly increasing chain $A_1\subsetneq A_2\subsetneq A_3\subsetneq \cdots$ of subsets of $\mbox{Fix}(F)$. Therefore, $T$ has the strict order property.

Let $\varphi(x,t)$ be the same formula. To see that $T$ has TP2, by compactness, we only need to show that given any $n\in\mathbb{N}$, there is some $(a_{ij})_{1\leq i,j\leq n}$ such that for any $1\leq i\leq n$, we have $\{\varphi(x,a_{ij}):1\leq j\leq n\}$ is 2-inconsistent and $\{\varphi(x,a_{if(i)}):1\leq i\leq n\}$ is consistent for any $f:\{1,\ldots,n\}\to \{1,\ldots,n\}$.

Given $n\in\mathbb{N}$, let $A_n\subseteq \mbox{Fix}(F)$ be a set with $n^n$-many elements. Fix a bijection $\eta:A_n\to \{1,\ldots,n\}^{\{1,\ldots,n\}}$ where $\{1,\ldots,n\}^{\{1,\ldots,n\}}$ is the set of all functions from $\{1,\ldots,n\}$ to itself. Let $(a_{ij})_{1\leq i,j\leq n}$ be such that $\varphi(x,a_{ij})$ codes the set $$B_{ij}:=\{a\in A_n:\eta(a)(i)=j\}\subseteq A_n.$$ For any $1\leq i\leq n$, as $B_{i1},\ldots,B_{in}$ form a complete partition of $A_n$, we get $\{\varphi(x,a_{ij}):1\leq j\leq n\}$ is 2-inconsistent. On the other hand, for any $f:\{1,\ldots,n\}\to\{1,\ldots,n\}$ the element $\eta^{-1}(f)\in A_n$ witnesses that $\{\varphi(x,a_{if(i)}):1\leq i\leq n\}$ is consistent.

Finally, as $(F,\mbox{Frob})$ interprets ultraproducts of initial segments of natural numbers with truncated addition and multiplication by Corollary \ref{cor-interval}, the undecidability follows from \cite[Section 4]{krynicki2005theories}.
\end{proof}

\subsection{Algebraic closure}
We now turn out attention to the study of the algebraic closure for a structure $(F,\mbox{Frob})\in\mathcal{S}$. Let $F$ be a pseudofinite field and $F^{alg}$ be the smallest algebraically closed field containing $F$. Take a tuple $a\in F$. Then the algebraic closure in the pseudofinite field $\mbox{acl}_{F}(a)$ is simply the algebraic closure in $F^{alg}$ intersected with $F$, i.e.\ $\mbox{acl}_{F}(a)=\mbox{acl}_{F^{alg}}(a)\cap F$. 

As ACFA is the model companion of the theory of difference fields, we can embed $(F,\mbox{Frob})$ into some $(K,\sigma)\models \text{ACFA}$. We might wonder if similarly, the algebraic closure in the theory of $(F,\mbox{Frob})$ is the same as the algebraic closure in $(K,\sigma)$ intersected with $F$, i.e.\ the algebraic elements are defined by difference polynomials. The following results provide a negative answer to this. 

\begin{theorem}\label{lem-defclosure}
For any $n>0$, there is some $(F,\mbox{Frob})\in\mathcal{S}$, an element $a_n\in F$and a tuple $b_n$ such that $a_n$ belongs to the definable closure of $b_n$ in $(F,\mbox{Frob})$, but $deg_{\sigma}(a_n/b_n)=n$.
\end{theorem}
We need a lemma first.
\begin{lemma}\label{lem-measure2n}
Let $\displaystyle{\varphi(x;y_1,\ldots,y_n):=\exists z(z^2=x+y_1)\land \bigwedge_{2\leq i\leq n}\forall z\neg(z^2=x+y_i)}$. There is $C_n\in\mathbb{R}^{>0}$ such that for any $\mathbb{F}_q$ with $\mbox{char}(\mathbb{F}_q)\neq 2$ and $b_1,\ldots,b_n$ distinct $n$-elements in $\mathbb{F}_q$, we have $$\left||\varphi(\mathbb{F}_q,b_1,\ldots,b_n)|-\frac{q}{2^n}\right|\leq C_n\cdot q^{\frac{1}{2}}.$$
\end{lemma}
\begin{proof}
Given distinct elements $b_1,\ldots,b_n\in\mathbb{F}_q$. Take an element $a\in \mathbb{F}_q$ such that $a$ is not a square. Let $J$ be the ideal in $\mathbb{F}_q[X,X_1,\ldots,X_n]$ generated by $$\{X_1^2-(X+b_1)\}\cup\{X_i^2-a(X+b_i):2\leq i\leq n\}.$$ By Fact \ref{duret4.3}, $J$ is absolutely prime, whence $V(J)$ is an absolutely irreducible variety of dimension 1. By the Lang-Weil estimate $$||V(J)\cap (\mathbb{F}_q)^{n+1}|-q|\leq N_n\cdot q^{\frac{1}{2}},$$ where $N_n$ is a constant only depends on the degree and dimension of the variety, which in our case is independent from $b_1,\ldots,b_n,a$ and $\mathbb{F}_q$ and only depends on $n$. Let $$\pi:V(J)\cap (\mathbb{F}_q)^{n+1}\to \mathbb{F}_q$$ be the projection on the the first coordinate. Clearly, $\pi$ is a $2^n$-to-one function. Therefore, $$|\varphi(\mathbb{F}_q,b_1,\ldots,b_n)|=|\pi(V(J)\cap (\mathbb{F}_q)^{n+1})|=\frac{1}{2^n}\cdot|V(J)\cap (\mathbb{F}_q)^{n+1}|.$$ Let $C_n:=\frac{N_n}{2^n}$. We conclude that $$\left||\varphi(\mathbb{F}_q,b_1,\ldots,b_n)|-\frac{q}{2^n}\right|\leq C_n\cdot q^{\frac{1}{2}}.$$
\end{proof}
Now we prove Theorem \ref{lem-defclosure}.
\begin{proof}
Given $n\in\mathbb{N}$, for each $p\in\mathbb{P}$, let $k_p\in\mathbb{N}$ be such that \begin{itemize}
\item
$k_p>\max\{f(p,p),14p^{n}\}$ where $f(p,p)$ is given by Equation \ref{eq1} in Definition \ref{def-phip};
\item
$n!$ divides $k_p$;
\item
$\frac{p^{k_p}}{2^{p^n}}>2 C_{p^n}\cdot p^{\frac{k_p}{2}}$.
\end{itemize}
Let $(F,\mbox{Frob}):=\prod_{p\in\mathbb{P}}(\mathbb{F}_{p^{k_p}},\mbox{Frob}_p)/\mathcal{U}$ where $\mathcal{U}$ is a non-principal ultrafilter on $\mathbb{P}$.
Clearly, $(F,\mbox{Frob})\in\mathcal{S}$ and $\mbox{Fix}(\sigma^n):=\{x\in F:\sigma^n(x)=x\}\neq \mbox{Fix}(\sigma^k)$ for any $k<n$.

Take an element $a_n\in \mbox{Fix}(\sigma^n)$ such that $deg_{\sigma}(a_n)=n$. Let $$\xi(x,a_n):=\exists z(z^2=a_n+x)\land\forall y(\sigma^n(y)=y\land (y\neq a_n\to \neg\exists z(z^2=y+x))).$$

As $k_p>14p^{n}$, for each prime $p\in\mathbb{N}$ we know by Theorem \ref{thm-sop} and the subsequent remark that $Y_n:=\xi((F,\mbox{Frob}),a_n)\neq\emptyset$. We claim that $\pmb{\delta}_F(Y_n)=1$. Suppose $a_n=(a_p)_{p\in\mathbb{P}}/\mathcal{U}$. For each $p\in\mathbb{P}$, let $a_p,b_1,\ldots,b_{p^n-1}$ be a list of all elements in $\mathbb{F}_{p^n}\subseteq \mathbb{F}_{p^{k_p}}$. Let $$\varphi(x,y_1,\ldots,y_{p^n}):=\exists z(z^2=x+y_1)\land \bigwedge_{2\leq i\leq p^n}\forall z\neg(z^2=x+y_i).$$
Note that for any $b\in\mathbb{F}_{p^{k_p}}$ we have $$\xi((\mathbb{F}_{p^{k_p}},\mbox{Frob}_p),a_p)=\varphi(\mathbb{F}_{p^{k_p}},a_p,b_1,\ldots,b_{p^n-1}).$$ By Lemma \ref{lem-measure2n}, $$||\varphi(\mathbb{F}_{p^{k_p}},a_p,b_1,\ldots,b_{p^n-1})|-\frac{p^{k_p}}{2^{p^n}}|\leq C_{p^n}\cdot p^{\frac{p^{k_p}}{2}},$$ for all $p>2$. Therefore, $$|Y_n|\geq \frac{p^{k_p}}{2^{p^n}}-C_{p^n}\cdot p^{\frac{p^{k_p}}{2}}>\frac{1}{2}\cdot\frac{p^{k_p}}{2^{p^n}}.$$ Since $$\lim_{p\to\infty}\frac{\log (p^{k_p}/2\cdot 2^{p^{n}})}{\log p^{k_p}}=1,$$ we get $\pmb{\delta}_F(Y_n)=1$.


Take an element $b_n\in Y_n$ such that $\pmb{\delta}_F(b_n)>0$. Note that $a_n\in \mbox{dcl}(b_n)$ and $\pmb{\delta}_F(a_n)=0$. Thus, using additivity of $\pmb{\delta}_F$, $$\pmb{\delta}_F(b_n/a_n)=\pmb{\delta}_F(a_n,b_n)-\pmb{\delta}_F(a_n)= \pmb{\delta}_F(b_n)+\pmb{\delta}_F(a_n/b_n)-\pmb{\delta}_F(a_n)=\pmb{\delta}_F(b_n)>0.$$ 

Therefore, $\mbox{SU}_{\text{ACFA}}(b_n/a_n)=\omega$. By our choice, we also have $\mbox{SU}_{\text{ACFA}}(b_n)=\omega$. Hence, $a_n$ is independent from $b_n$ in $(\tilde{F},\mbox{Frob})$. Again, by our choice, $deg_{\sigma}(a_n)=n$. But if $deg_{\sigma}(a_n/b_n)<n$, then $a_n$ and $b_n$ will not be independent in $(\tilde{F},\mbox{Frob})$ in the theory of ACFA. We conclude that $deg_{\sigma}(a_n/b_n)=n$ and $a_n$ is in the definable closure of $b_n$. 
\end{proof}

\subsection{Further remarks:}
We conclude this paper with some remarks.
\begin{enumerate}
\item
As we have mentioned in the remark after Theorem \ref{th1}, we can easily generalise the results of this paper to other classes, provided the fields grow fast enough. Let $(F,\sigma):=\prod_{i\in I}(\mathbb{F}_{{p_i}^{k_i}},\mbox{Frob}_{{p_i}^{m_i}})/\mathcal{U}$, with $p_i^{k_i}>>p_i^{m_i}$ for all $i\in I$, then all the results in Section \ref{sec2} and Section \ref{sec3} are true for $(F,\sigma)$ as well. Corollary \ref{cor-interval} will also be true if the fixed field of $(F,\sigma)$ is infinite. However, if $(F,\sigma):=\prod_{i\in I}(\mathbb{F}_{{p}^{k_i}},\mbox{Frob}_{{p}^{m_i}})/\mathcal{U}$ with $k_i$ and $p_i$ coprime for all $i\in I$, then it is not clear whether its theory will always be wild.
\item
As pointed out by Hrushovski,\footnote{Personal communication.} Conjecture \ref{conj} is related to the question of whether, for a given difference equation $E(x)=0$, the Frobenius specialization, $E_p(x)=0$ is absolutely irreducible for almost all $p\in\mathbb{P}$. If it is the case, then $E_p(x)=0$ will have enough solutions in a large finite field of characteristic $p$ by the Lang-Weil estimate. Consequently, $E(x)=0$ will have the same set of solutions in a large finite field with the Frobenius map $\mbox{Frob}_p$, whence Conjecture \ref{conj} would be true. However, it is not always the case that the Frobenius specialization is absolutely irreducible for almost all $p$.

\end{enumerate}

\label{Bibliography}

\bibliographystyle{plain} 
\bibliography{dim11.bib}

\begin{thebibliography}{10}

\bibitem{ax1968elementary}
James Ax.
\newblock The elementary theory of finite fields.
\newblock {\em Annals of Mathematics}, 88(2):239--271, 1968.

\bibitem{BaysBreuillard}
Martin Bays and Emmanuel Breuillard.
\newblock Projective geometries arising from {E}lekes-{S}zab\'{o} problems.
\newblock {\em arXiv preprint: 1806.03422}, 2018.

\bibitem{cafure}
Antonio Cafure and Guillermo Matera.
\newblock Improved explicit estimates on the number of solutions of equations
  over a finite field.
\newblock {\em Finite Fields and Their Applications}, 12(2):155--185, 2006.

\bibitem{zoeND}
Zo{\'e} Chatzidakis.
\newblock Model theory of difference fields.
\newblock In {\em The Notre Dame Lectures}, P. Cholak, series editor, Lecture
  notes in Logic, pages 45--96. A.K. Peters, 2005.

\bibitem{chatzidakis2015model}
Zo{\'e} Chatzidakis.
\newblock Model theory of fields with operators - a survey.
\newblock In {\AA}sa Hirvonen, Juha Kontinen, Roman Kossak, and Andr{\'e}s
  Villaveces, editors, {\em Logic Without Borders: Essays on Set Theory, Model
  Theory, Philosophical Logic and Philosophy of Mathematics}, pages 91--114.
  Walter de Gruyter GmbH \& Co KG, 2015.

\bibitem{chatzidakis1999model}
Zo{\'e} Chatzidakis and Ehud Hrushovski.
\newblock Model theory of difference fields.
\newblock {\em Transactions of the American Mathematical Society},
  351(8):2997--3071, 1999.

\bibitem{chatzidakis2002model}
Zo{\'e} Chatzidakis, Ehud Hrushovski, and Ya'acov Peterzil.
\newblock Model theory of difference fields, {II}: {P}eriodic ideals and the
  trichotomy in all characteristics.
\newblock {\em Proceedings of the London Mathematical Society}, 85(2):257--311,
  2002.

\bibitem{chatzidakis1998generic}
Zo{\'e} Chatzidakis and Anand Pillay.
\newblock Generic structures and simple theories.
\newblock {\em Annals of Pure and Applied Logic}, 95(1-3):71--92, 1998.

\bibitem{chatzidakis1992definable}
Zo{\'e} Chatzidakis, Lou van~den Dries, and Angus Macintyre.
\newblock Definable sets over finite fields.
\newblock {\em Journal f{\"u}r die reine und angewandte Mathematik},
  427:107--135, 1992.

\bibitem{cohn1965difference}
Richard Cohn.
\newblock {\em Difference algebra}.
\newblock Interscience Publishers, 1965.

\bibitem{duret}
Jean-Louis Duret.
\newblock Les corps alg{\'e}briquement clos non s{\'e}parablement clos ont la
  propri{\'e}t{\'e} d'ind{\'e}pendance.
\newblock In Leszek Pacholski, Jedrzej Wierzejewski, and Alec~J. Wilkie,
  editors, {\em Model Theory of Algebra and Arithmetic}, pages 136--162.
  Springer, Berlin, Heidelberg, 1980.

\bibitem{garcia2015pseudofinite}
Dar{\'\i}o Garc{\'\i}a, Dugald Macpherson, and Charles Steinhorn.
\newblock Pseudofinite structures and simplicity.
\newblock {\em Journal of Mathematical Logic}, 15(01):1550002, 2015.

\bibitem{hrushovski1996mordell}
Ehud Hrushovski.
\newblock The {M}ordell-{L}ang conjecture for function fields.
\newblock {\em Journal of the American Mathematical Society}, 9(3):667--690,
  1996.

\bibitem{hrushovski2004elementary}
Ehud Hrushovski.
\newblock The elementary theory of the {F}robenius automorphisms.
\newblock {\em arXiv preprint: math/0406514}, 2004.

\bibitem{hrushovski2012stable}
Ehud Hrushovski.
\newblock Stable group theory and approximate subgroups.
\newblock {\em Journal of the American Mathematical Society}, 25(1):189--243,
  2012.

\bibitem{hrushovski2013pseudo}
Ehud Hrushovski.
\newblock On pseudo-finite dimensions.
\newblock {\em Notre Dame Journal of Formal Logic}, 54(3-4):463--495, 2013.

\bibitem{hrushovski1994groups}
Ehud Hrushovski and Anand Pillay.
\newblock Groups definable in local fields and pseudo-finite fields.
\newblock {\em Israel Journal of Mathematics}, 85(1):203--262, 1994.

\bibitem{hrushovski2008counting}
Ehud Hrushovski and Frank Wagner.
\newblock Counting and dimensions.
\newblock {\em London Mathematical Society Lecture Notes Series}, 350:161,
  2008.

\bibitem{krynicki2005theories}
Micha{\l} Krynicki and Konrad Zdanowski.
\newblock Theories of arithmetics in finite models.
\newblock {\em The Journal of Symbolic Logic}, 70(1):1--28, 2005.

\bibitem{macpherson2008one}
Dugald Macpherson and Charles Steinhorn.
\newblock One-dimensional asymptotic classes of finite structures.
\newblock {\em Transactions of the American Mathematical Society},
  360(1):411--448, 2008.

\bibitem{ryten2007model}
Mark Ryten.
\newblock {\em Model theory of finite difference fields and simple groups}.
\newblock PhD thesis, School of Mathematics, University of Leeds, 2007.

\bibitem{ryten2006acfa}
Mark Ryten and Ivan Toma{\v{s}}i{\'c}.
\newblock {ACFA} and measurability.
\newblock {\em Selecta Mathematica}, 11(3):523--537, 2006.

\bibitem{shelah1990classification}
Saharon Shelah.
\newblock {\em Classification theory and the number of non-isomorphic models},
  volume~92.
\newblock Elsevier, 1990.

\end{thebibliography}

\end{document}